\documentclass[11pt]{amsart}
\usepackage{amscd, amsmath, amsthm, amssymb, xy, xypic}
\usepackage{graphicx}
\graphicspath{ {images/} }
\usepackage{tikz}
\usetikzlibrary{arrows,decorations.pathmorphing,backgrounds,positioning,fit,petri}
\usepackage{tabto} % tab
\usepackage{longtable}

\usepackage{tikz}
\usepackage{tikz-3dplot}
\usetikzlibrary{calc}

\newtheorem{theorem}{Theorem}[section]
\newtheorem{lemma}[theorem]{Lemma}
\newtheorem{proposition}[theorem]{Proposition}

\newtheorem{condition}[theorem]{Condition}
\newtheorem{corollary}[theorem]{Corollary}
\theoremstyle{definition}     % italic or bold etc.
\newtheorem{definition}[theorem]{Definition}

\newtheorem{example}[theorem]{Example}

\numberwithin{equation}{section}

\def \hd #1 {\bfseries #1  \mdseries}
\def \italic #1 {\bfseries \it #1 \rm \mdseries}
\def \ra {\rightarrow}
\def \cen #1 { \begin{center} #1 \end{center}}

\def \mbz {\mathbb Z}

\def \mbc {\mathbb C}
\def \mbp {\mathbb P}

\def \mbq  {\mathbb {Q}}
\def \mco  {\mathcal {O}}

\def \mcX  {\mathcal {X}}
\def \mcD  {\mathcal {D}}

\def \Pic {{\rm{Pic}}}
\def \rk {{\rm{rk}}}
\def \im {{\rm{im}}}

\def \Gr {{\rm{Gr}}}

\def \rank {{\rm{rank}}}
\def \NE {{\rm{NE}}}

\begin{document}
\title
{$d$-semistable Calabi--Yau threefolds of Type III }

\author{Nam-Hoon Lee}
\address{
Department of Mathematics Education, Hongik University
42-1, Sangsu-Dong, Mapo-Gu, Seoul 121-791, Korea
}
\email{nhlee@hongik.ac.kr}
\address{School of Mathematics, Korea Institute for Advanced Study, Dongdaemun-gu, Seoul 130-722, South Korea }
\email{nhlee@kias.re.kr}
%\thanks{Partially supported by NSF grant DMS-0093542}
%\thanks{The author is grateful for several valuable comments from the referee.}
%\keywords{Calabi--Yau manifold, Polytopal log Calabi--Yau manifold, Triangulation}
\begin{abstract}
We develop some methods to construct normal crossing varieties whose dual complexes are two-dimensional, which are smoothable to Calabi--Yau threefolds. We calculate topological invariants of smoothed Calabi--Yau threefolds and show that several of them are new examples.
\end{abstract}
\maketitle
%%%%%%%%%%%%%%%%%%%%%%%%%%%%%%%%%%%%%%%%%%%
%\tableofcontents
% \setcounter{section}{-1}
\section{Introduction}
A Calabi-–Yau manifold is a compact K\"ahler manifold with trivial canonical class
such that the intermediate cohomologies of its structure sheaf are all trivial
($h^i (M,\mco_M ) = 0$ for $0 < i < \dim(M)$).
Calabi--Yau threefolds have attracted much interest from both  of mathematics and physics but the classification  of Calabi--Yau threefolds is widely open. Even boundedness of their Hodge numbers is still unknown.
Thus developing method of constructing Calabi--Yau threefolds and finding new examples are
  of interest.
   If a normal crossing variety  is the central fiber of a semistable degeneration of a Calabi--Yau manifolds,  it can be regarded as a member in a deformation family of those Calabi--Yau manifolds.
 Semistable degenerations of $K3$ surfaces, which are Calabi--Yau twofolds, have been investigated by several authors (\cite{Ku, PePi}).
 V.\ Kulikov proved that any degeneration of $K3$ surfaces can be
  modified to be semistable one whose total space has trivial canonical divisor and he also classified the central fibers into three types.
 R.\ Friedman proved a smoothing theorem for
a $d$-semistable normal crossing variety to $K3$ surfaces (\cite{Fr}), which is a converse of V.\ Kulikov's result. More concretely he found a sufficient condition for a normal crossing variety to be a central fiber of semistable degeneration of $K3$ surfaces.
A smoothing theorem for the higher dimensional case has been introduced by Y.\ Kawamata and Y.\ Namikawa in \cite{KaNa}.
A remarkable difference between two-dimensional cases of $K3$ surfaces and higher dimensional cases is that there are multiple deformation types for higher dimensional Calabi--Yau manifolds. So building a  normal crossing variety smoothable to a Calabi--Yau manifold can be regarded as building a deformation type of Calabi--Yau manifolds.

Smoothing of normal crossing variety with two components to Calabi--Yau threefolds   has been actively investigated.
A.\ Tyurin studied Calabi--Yau threefolds  via degeneration to normal crossing varieties, in particular 2-component varieties glued along a common anticanonical divisor in his posthumous paper \cite{Ty}.
R.\ Donagi, S.\ Katz and M.\ Wijnholt   considered  a holographic relation between F-theory on a degenerate Calabi-Yau and a dual theory on its boundary in \cite{DoShMa}.
A possible mirror construction of Calabi--Yau manifolds by smoothing has been suggested by C.\ Doran, A.\ Harder, A.\ Thompson (\cite{DoHaTh}) and materialized more in \cite{Lee1}.
In this case the dual complexes of the normal crossing varieties are one-dimensional (line segments).
In this note, we consider normal crossing varieties, smoothable to Calabi--Yau threefolds, whose dual complexes are two-dimensional.
There are some difficulties in making such normal crossing varieties that do not appear in the case of normal crossing varieties with only two components.
 We  explain this with examples and demonstrate how to circumvent the difficulties. We also develop a method for calculating their topological invariants such as Hodge numbers.
It turns out that several of them  have Hodge numbers different from Calabi--Yau threefolds from toric setting. It is also notable that the methods in this note naturally leads to the construction of multiple non-homeomorphic Calabi--Yau threefolds with the same Hodge numbers.

The structure of this note is as follows.

Section \ref{sec2} is a background section for smoothing and degeneration of Calabi--Yau manifolds.
We introduce basic definitions and  the smoothing theorem of Kawamata-Namikawa, which is a main tool of the construction of Calabi--Yau manifolds in this note.

In Section \ref{sec3}, we concentrate on the case of normal crossing varieties of three components. Some formulas for Hodge numbers of their smoothing are developed.

Section \ref{sec4} is devoted to the $d$-semitablility condition.
Considering a concrete example, we demonstrate the difficulty in making $d$-semistable normal crossing varieties when they have three components and how to  circumvent the difficulty.

In  Section \ref{sec5}, we generalize and systemize the procedure in Section \ref{sec4}, stating exact conditions for the construction of $d$-semistable normal crossing varieties.

Applying our methods, we give several examples of Calabi--Yau threefolds of type III in  Section \ref{sec6}. We consider six configurations which together produce more than fifty examples, including ones with new Hodge numbers. We also introduce a  different kind of   Calabi--Yau threefold of type III, which demonstrates that there are some room for generalization of  our method.

\section{Degeneration and smoothing of Calabi--Yau manifolds} \label{sec2}

A normal crossing
variety is a reduced complex analytic space which is locally isomorphic to a normal
crossing divisor on a smooth variety. It is said to be simple if all of its components are smooth varieties. In this note, we only consider simple ones.
Let $X$ be a normal crossing variety with irreducible components $\{X_i | i \in I \}$.
A stratum $S$ of $X$ is any irreducible component of an intersection $\bigcap_{i \in J} X_i$ for some $J \subset I$.
The dual complex $\mcD(X)$ of $X$ is a simplicial complex whose vertices are labeled by the irreducible components of $X$ and for every stratum of dimension $r$ we attach a ($\dim X -r$)-dimensional simplex.

A semistable degeneration is a proper flat holomorphic map $\varphi: \mathcal X \ra \Delta$ from a  K\"ahler manifold $\mathcal X$ onto the complex unit disk $\Delta$
such that the fiber $\mathcal X_t = \varphi^{-1}(t)$ is a smooth complex variety for $t \neq 0$ and the central fiber $\mathcal X_0$ is a simple normal crossing divisor of $\mathcal X$.
Following naming in \cite{Fr,Ku}, we give following definition:

\begin{definition}
A projective normal crossing variety of dimension $n$ is called a {$d$-semistable Calabi--Yau $n$-fold of type $k+1$} if it has trivial dualizing sheaf, it is the central fiber in a semistable degenerations of Calabi--Yau manifolds and its dual complex is $k$-dimensional.
\end{definition}
It seems natural to include some mild singularities for Calabi--Yau manifolds of dimension higher than two, but in this note we stick to only smooth ones.
In this definition, the usual Calabi--Yau $n$-fold  is  a $d$-semistable Calabi--Yau $n$-fold of type I.
We also say that $\mcX_0$ is smoothable to $\mcX_t$($t\neq 0$) with smooth total space $\mcX$.
In this note, we mainly consider   $d$-semistable Calabi--Yau threefolds of type III that are composed of three components.
We also briefly discuss  examples of Calabi--Yau threefolds of type III that have four or more components.
%By topological invariants of a $d$-semistable Calabi--Yau threefold, we mean those of the Calabi--Yau %manifold which is a general fiber in the degeneration.
% It does not satisfies the conditions of Theorem \ref{kana}.
Our main tool is the smoothing theorem of Y. Kawamata and Y. Namikawa (\cite{KaNa}), which is stated below for the readers' convenience.
\begin{theorem}[Y. Kawamata, Y. Namikawa] \label{kana}Let $Y =
\bigcup_i Y_i$ be a compact normal crossing variety of dimension $n$ such that
\begin{enumerate}
\item It is  K\"ahler and $d$-semistable.
\item Its dualizing sheaf is trivial: $\omega_{Y} = \mco_{Y}$,
\item $H^{n-2}  (Y_i, \mco_{Y_i}) = 0$ for any $i$ and $H^{n-1}  (Y, \mco_{Y}) = 0$.
\end{enumerate}
Then $Y$ is smoothable to an  $n$-fold $M_Y$  with  trivial canonical class and the total
space of smoothing is smooth.
\end{theorem}
It was showed that the Hodge numbers of the Calabi--Yau manifolds can be calculated from the geometry of the normal crossing varieties (\cite{Leeth,Lee}).

Let us take a very simple example.
\begin{example}\label{kanaex}
Let $W_1$, $W_2$ be copies of $\mbp^3$ and $D$ be a smooth quartic surface in $\mbp^3$.
Then $W_1$, $W_2$ contain copies of $D$.
Let $W = W_1 \cup_D W_2$, where `$\cup_D$' means gluing along $D$.
Then the normal crossing variety $W$ is projective and has trivial dualizing sheaf but it is not $d$-semistable.
In \cite{KaNa},  $W_2$ is blowed up along a smooth curve $c$ in the linear system $|\mco_D(8)|$ to become $\widetilde W_2$. The proper transform $\widetilde D$  in $\widetilde W_2$ of $D$ is isomorphic to $D$. So one can paste $W_1$ and $\widetilde W_2$ along $D$ and $\widetilde D$ to get a $d$-semistable normal crossing variety $\widetilde W$.
 It is smoothable to a Calabi--Yau threefold whose invariants are calculated in \cite{Lee}.
Note that  $\widetilde W$ is a $d$-semistable Calabi--Yau threefold of type II.
\end{example}

\section{The case of type III} \label{sec3}

Consider a normal crossing variety $Y=Y_1 \cup Y_2 \cup Y_3$. Let $Y_{ij} = Y_i \cap Y_j, Y_{ijk}=Y_i \cap Y_j \cap Y_k$.
 If we treat the double locus  $Y_{ij}$ and the triple locus $Y_{ijk}$ as subvarieties or divisors of  $Y_j$ and $Y_{jk}$, we denote them by $Y_{(ij)}$ and $Y_{(ijk)}$ respectively. Let $D_1 = Y_{23}, D_2 = Y_{31}, D_3 = Y_{12}$ and $\tau=Y_{123}$.
We want to smooth $Y$ to a Calabi--Yau manifold, using Theorem \ref{kana}.
Throughout this note, we assume  that $Y$ satisfies the following conditions, which contains all conditions except the $d$-semistability in Theorem \ref{kana}.

\begin{condition} \label{refcond}
\begin{enumerate}
\item $n=\dim Y \ge 3$
\item $H^{a}  (Y_i, \mco_{Y_i}) = 0$ and $H^{a}  (D_{i}, \mco_{D_i}) = 0$
for  each $i=1,2,3$ and $a=1,2$. $D_i$'s and $\tau$ are all connected.
\item There is an ample divisor $H_i$ of $Y_i$ such that $H_i|_{Y_{ij}} \sim H_j|_{Y_{ij}}$ for every $i, j$.
\item For each fixed $j$, $-\sum_{i \neq j} Y_{(ij)}$ is a canonical divisor of $Y_j$.
\end{enumerate}
\end{condition}

Note that the pair $(Y_i, Y_{(ji)}\cup Y_{(ki)})$ is a log Calabi--Yau pair for $\{i,j,k \}= \{1,2,3 \}$ and $Y_{(ji)}\cup Y_{(ki)}$ is a normal crossing of two rational surfaces if $\dim Y = 3$.

Now assume that $Y$  is smoothable to a Calabi--Yau manifold $M_Y$ with smooth total space and let us calculate topological invariants of $M_Y$.
In \S 7 of \cite{Lee}, we have introduced $G^{2i}(Y, \mbz)$ as a subgroup of
$\bigoplus_\alpha H^{2i}(Y_\alpha, \mbz)$, that is, the image of the map
$$H^{2i}(Y, \mbz) \ra \bigoplus_\alpha H^{2i}(Y_\alpha, \mbz).$$
We also located the Chern class of $M_Y$  as an element of $\bigoplus_{i}G^{2i}(Y, \mbz)$.
The group $G^{2i}(Y, \mbz)$ inherits the cup product from those of   $\bigoplus_\alpha H^{2i}(Y_\alpha, \mbz)$ with the mixed terms set to be zero.

It is (Definition  7.1 in \cite{Lee}):
\begin{align}\label{chern}
c(Y) = \sum_i \left( 1^{(i)} - \sum_{j (\neq i )} Y_{(j i)} \right ) c(Y_i),
\end{align}
where $1^{(i)}$ is the generator of $H^0(Y_i, \mbz)$.
The cup product of  $c(Y)$ with $\bigoplus_{i}G^{2i}(Y, \mbz)$ gives some information about
the cup product of  $c(M_Y)$ with $\bigoplus_{i}H^{2i}(M_Y, \mbz)$.
Using this formula, one can calculate the topological Euler characteristic of $M_Y$.
\begin{proposition} \label{eulereqn} If $Y$ is smoothable to a Calabi--Yau manifold $M_Y$ with smooth total space, then
$$e(M_{Y}) = e(Y_1) + e(Y_2) + e(Y_3) - 2 \left(e(D_1) + e(D_2) + e(D_3)\right) + 3e(\tau).$$
\end{proposition}
\begin{proof}
From equation \ref{chern},
$$e(M_{Y}) = c_3 (Y) =  \sum_i c_3(Y_i) - \sum_{i\neq j}  Y_{(j i)} \cdot c_2 (Y_i).$$
By the adjunction formula, for $\{i,j,k\} = \{1,2,3\}$,
$$Y_{(j i)} \cdot c_2 (Y_i)  = c_2(Y_{j i}) + Y_{(j i)}|_{Y_{ji}} \cdot c_1(Y_{ji}).$$
Since $Y$ is $d$-semistable, $N(Y_{ij})=0$, i.e.\  $Y_{(j i)}|_{Y_{ji}} + Y_{(ij)}|_{Y_{ij}} = - Y_{(kij)}$.
So we have
\begin{align*}
Y_{(j i)} \cdot c_2 (Y_i) + Y_{(i j)} \cdot c_2 (Y_j) &= \left(c_2(Y_{j i}) + Y_{(j i)}|_{Y_{ji}}\cdot c_1(Y_{ji}) \right ) +  \left(c_2(Y_{i j}) + Y_{(ij)}|_{Y_{ij}} \cdot c_1(Y_{ij}) \right )\\
&=c_2(Y_{j i}) + c_2(Y_{i j})  + \left ( Y_{(j i)}|_{Y_{ji}} + Y_{(ij)}|_{Y_{ij}} \right ) \cdot c_1(Y_{ij})\\
&=e(Y_{j i}) + e(Y_{i j})  + (-Y_{123}) \cdot c_1(Y_{ij})\\
&= 2e(Y_{ij}) - Y_{(kij)}^2 - c_1(Y_{kij}) \\
&=  2e(D_k) - Y_{(kij)}^2 - e(\tau).
\end{align*}
Hence
\begin{align*}
e(M_{Y}) &=  \sum_i c_3(Y_i) - \sum_{i\neq j}  Y_{(j i)} \cdot c_2 (Y_i)\\
         &=\sum_i e (Y_i) -  \sum_{i< j}  ( Y_{(j i)} \cdot c_2 (Y_i) + Y_{(i j)} \cdot c_2 (Y_j))\\
         &=\sum_i e (Y_i) - 2\sum_k e(D_k) +3 e(\tau) -(Y_{(312)}^2  + Y_{(123)}^2 + Y_{(213)}^2).
\end{align*}
Note
\begin{align*}
 0=N(Y_{12})|_{Y_{123}} + N(Y_{23})|_{Y_{123}} + N(Y_{13})|_{Y_{123}} &= 3 (Y_{(312)}^2  + Y_{(123)}^2 + Y_{(213)}^2).
\end{align*}
Therefore we have the formula.
\end{proof}

Next, we determine the Hodge number $h^{1,1}(M_Y)$.
\begin{proposition} \label{prop11}
There is  an exact sequence
$$0 \ra H^2(Y, \mbz) \stackrel{\eta}{\ra} H^2(Y_1, \mbz) \oplus  H^2(Y_2, \mbz) \oplus  H^2(Y_3, \mbz)  \stackrel{\mu}{\ra} H^2(D_1, \mbz) \oplus H^2(D_2, \mbz) \oplus H^2(D_3, \mbz),$$
where $\eta$ is the restriction map and the map $\mu$ is defined by
$$\mu(H_1, H_2, H_3) = (H_2|_{D_1}-H_3|_{D_1}, H_3|_{D_2}-H_1|_{D_2}, H_1|_{D_3}-H_2|_{D_3}).$$
\end{proposition}

\begin{proof}
Using the exponential sequence, there are isomorphisms
$$\Pic(Y) \simeq H^2(Y, \mbz), \Pic(Y_i) \simeq H^2(Y_i, \mbz), \Pic(D_i) \simeq H^2(D_i, \mbz).$$
Hence it is enough to show that the following sequence is exact:
$$0 \ra \Pic(Y){\ra} \Pic(Y_1) \oplus  \Pic(Y_2) \oplus  \Pic(Y_3) {\ra} \Pic(D_1) \oplus\Pic(D_2) \oplus\Pic(D_3).$$
The exactness of the above sequence comes from Proposition 2.6 of \cite{Fu}.

\end{proof}
 The subgroup $G^2(Y, \mbz)$ of
$$ H^2(Y_1, \mbz) \oplus  H^2(Y_2, \mbz) \oplus  H^2(Y_3, \mbz)$$
is defined by
$G^2(Y, \mbz) = \im (\eta)$ (p.\ 704 of \cite{Lee}) and now $G^2(Y, \mbz) = \ker \mu$ by Proposition \ref{prop11}.
It inherits the cup products from those of $ H^2(Y_1, \mbz) \oplus  H^2(Y_2, \mbz) \oplus  H^2(Y_3, \mbz)$, where the mixed terms are defined to be zero.
Let $NG^2(Y) = \langle e_1, e_2 \rangle$ be a subgroup of $ H^2(Y_1, \mbz) \oplus  H^2(Y_2, \mbz) \oplus  H^2(Y_3, \mbz)$,  where
$$e_1 = (-Y_{(21)}-Y_{(31)}, Y_{(12)}, Y_{(13)}), e_2 = ( Y_{(21)}, -Y_{(12)}-Y_{(32)}, Y_{(23)}).$$
Then it was showed in \S 5 of \cite{Lee} that $NG^2(Y)$ is a subgroup of $G^2(Y)$ that is degenerated with respect to the cup product and $\rk(NG^2(Y, \mbz))=2$.
Furthermore, it is known that there is an injection (Proposition 5.4 in \cite{Lee}):
\begin{align}\label{gmap}
\left ( G^2(Y, \mbz)/ NG^2(Y, \mbz) \right )_f \ra H^2(M_Y, \mbz)_f
\end{align}
with finite index, where $A_f$ for an Abelian group $A$ means its quotient by torsion part.
This injection preserves the cup product, so one can use it to calculate the cup product on $ H^2(M_Y, \mbz)$ (See \cite{Lee} and \S 6 of \cite{Lee1} for more details).
Since  $G^2(Y, \mbz) = \ker \mu$,  we can easily calculate $G^2(Y, \mbz)$.

\begin{corollary}\label{hodge} If $Y$ is smoothable to a Calabi--Yau  manifold $M_Y$ with smooth total space, then
\begin{align*}
h^{1,1} ( M_Y ) =\dim ( \ker   (H^2(Y_1, \mbz) &\oplus  H^2(Y_2, \mbz) \oplus  H^2(Y_3, \mbz)\\
 & \stackrel{\mu}{\ra} H^2(Y_{12}, \mbz) \oplus H^2(Y_{23}, \mbz) \oplus H^2(Y_{31}, \mbz) )  )-2.
\end{align*}
\end{corollary}
\begin{proof}
From Theorem 4.3 in \cite{Lee}, we note
$$h^{1,1}(M_Y) = h^2(M_Y) = h^2(Y) - 3 + 1 = h^2(Y)-2.$$
Hence we are done by Proposition \ref{prop11}.
\end{proof}

For a Calabi--Yau threefold $M_Y$, $h^{1,2}(M_Y) = h^{1,1}(M_Y) - \frac{1}{2} e(M_Y)$, So Proposition \ref{eulereqn} and Corollary \ref{hodge} determine all the Hodge numbers of $M_Y$.

\section{$d$-semistability and an example} \label{sec4}

Consider a normal crossing variety $Y=Y_1 \cup Y_2 \cup Y_3$, satisfying Condition \ref{refcond}.
By Theorem \ref{kana}, one can show that $Y$ is smoothable to a Calabi--Yau manifold of dimension $n$  if it is $d$-semistable.
The $d$-semistability (also called as `logarithmic structure' in \cite{KaNa}) is the condition that the normal crossing variety $Y$ is the central fiber in semistable degeneration.
Suppose that a normal crossing variety $X=X_1 \cup X_2 \cup X_3$ is the central fiber in a semistable degeneration $\varphi:\mathcal X \ra \Delta$.
Note
$$X|_{X_j} = \mathcal X_0|_{X_j} \sim \mathcal X_t|_{X_j} = 0 $$
on $X_j$, where $t \neq 0$.
We have  $(X_j|_{X_j})|_{X_{ij}} = (X_j|_{X_i})|_{X_{ij}}$ in $\Pic(X_{ij}) $ since $X_{ij}$ is subvariety of both of $X_i, X_j$.
For distinct $i, j, k$ in $\{1,2,3\}$,
\begin{align*}
0=(X|_{X_j})|_{X_{ij}} &=  \left( (X_i + X_j + X_k) |_{X_j} \right )|_{X_{ij}}\\
                         &=(X_i|_{X_j})|_{X_{ij}}+(X_j|_{X_j})|_{X_{ij}}+ (X_k|_{X_j})|_{X_{ij}} \\
                         &=  X_{(ij)}|_{X_{ij}} + (X_j|_{X_i})|_{X_{ij}}+  {X_{(kij)}}\\
                         &=  X_{(ij)}|_{X_{ij}} + X_{(ji)}|_{X_{ij}}+  X_{(kij)}
\end{align*}
in $\Pic(X_{ij})$.

So let
\begin{equation}   \label{cnormal}
  N_Y(Y_{ij}) =  Y_{(ij)}|_{Y_{ij}} + Y_{(ji)}|_{Y_{ij}}+  Y_{(kij)} \in \Pic(Y_{ij}).
\end{equation}

Noting $D_i = Y_{jk}$ for $\{i,j,k \} = \{1,2,3 \}$,
let us call
$$(N_Y(D_1),N_Y(D_2),N_Y(D_3))  \in  \Pic(D_1) \oplus \Pic(D_2) \oplus \Pic(D_3)$$
the collective normal class of $Y$.
 The triviality of the collective normal class of $Y$ is called the `triple point formula' in two dimensional case (\cite{Pe}) and is a necessary condition for  the $d$-semistability. On the other hand, if $\mathcal {E}xt_Y^1 ( \Omega_Y^1, \mco_Y ) |_D$ is trivial, then $Y$ is said to be $d$-semistable (\cite{Fr}), where $D=D_1 \cup D_2 \cup D_3$.
 The triviality of the collective normal class is equivalent with the $d$-semistability in our case.
\begin{proposition} \label{dsemiprop}
 The triviality of the collective normal class of $Y$ implies the $d$-semistability of $Y$. Hence if the collective normal class of $Y$ is trivial, then $Y$ is a Calabi--Yau manifold of type III.
\end{proposition}
\begin{proof}
Note that $\mathcal {E}xt_Y^1 ( \Omega_Y^1, \mco_Y ) |_{D}$ is a line bundle on $D$ and
$\mathcal {E}xt_Y^1 ( \Omega_Y^1, \mco_Y ) |_{D_i} = N_Y(D_i)$ (\cite{Fr}).
Suppose that $Y$ has a trivial  collective normal class, i.e.\ $\mathcal {E}xt_Y^1 ( \Omega_Y^1, \mco_Y ) |_{D_i} = N_Y(D_i) =\mco_{D_i}$ for any $i \neq j$.
In order  to show $\mathcal {E}xt_Y^1 ( \Omega_Y^1, \mco_Y ) |_{D}$ =0,
it is enough to show  that  the map
$$\eta:\Pic({D_1 \cup D_2 \cup D_3}) \ra  \Pic(D_1)\oplus  \Pic(D_2)\oplus  \Pic(D_3)$$
 is injective, where $\eta$ is the restriction map.
 Consider the exact sequence of sheaves of Abelian groups:
$$1 \ra \mco^*_{ D_1 \cup D_2 } \ra   \mco^*_{ D_1} \times \mco^*_{D_2}  \ra  \mco^*_{ D_1 \cap D_2} \ra 1$$
to deduce a long exact sequence
$$1 \ra \mbc^* \ra \mbc^* \times \mbc^* \stackrel{v}{\ra} \mbc^* {\ra} \Pic(D_1 \cup D_2 ) \stackrel{\lambda}{\ra}  \Pic(D_1 ) \oplus \Pic( D_2) $$
because $D_1 \cup D_2$, $ D_2$, and $ D_1 \cap D_2 = \tau$ are all connected.
Since the map $v$ is surjective, the map $\lambda$ is injective.
Consider again the exact sequence of sheaves of Abelian groups:
$$1 \ra \mco^*_{ D_1 \cup D_2 \cup  D_3 } \ra   \mco^*_{  D_1 \cup D_2  } \times \mco^*_{D_3}  \ra  \mco^*_{ (D_1 \cup D_2) \cap D_3} \ra 1$$
to get another long exact sequence
$$1 \ra \mbc^* \ra \mbc^* \times \mbc^* \stackrel{v'}{\ra} \mbc^* {\ra} \Pic(D_1 \cup D_2 \cup D_3 ) \stackrel{\lambda'}{\ra}  \Pic(D_1 \cup D_2) \oplus \Pic( D_3) $$
because $D_1 \cup D_2 \cup  D_3$, $  D_1 \cup D_2$, and $(D_1 \cup D_2) \cap D_3 = \tau$ are all connected.
The map $v'$ is surjective and so the map $\lambda'$ is injective.
The injectivenesses of $\lambda'$ and $\lambda$ imply that of $\eta$.
\end{proof}

It is relatively easy to find a normal crossing variety that satisfies Condition \ref{refcond}  only  but is not $d$-semistable.  One usually needs to blow up along some suitable divisors of its double loci  to make it $d$-semistable as in Example \ref{kanaex}.
If there are triple loci in the normal crossing varieties, the construction  gets quite complicated. We devote the rest of this section to making $Y$ of the following example $d$-semistable.
\begin{example} \label{quintic}
Consider  a normal crossing $Y$ of two hyperplanes $Y_1, Y_2$ and one cubic threefold $Y_3$ in $\mbp^4$. We all know that $Y$ is smoothable to a quintic Calabi--Yau threefold but the total space is not smooth.
We want to modify $Y$ so that we can apply Theorem \ref{kana}.
 It has trivial dualizing sheaf but is not $d$-semistable ---   its collective normal class is a divisor class $(\mco_{D_1}(5), \mco_{D_2}(5), \mco_{D_3}(5))$. Its dual complex is two-dimensional.
\end{example}
 As in Example \ref{kanaex}, we may need to choose smooth curves $C_{i}$ in the linear system $|\mco_{D_i}(5)|$  on $D_i$ and do blow-ups.
We assume for simplicity that $C_i$'s are all disjoint.
Then  blow up  $Y_1, Y_2, Y_3$ along a smooth curve $C_3, C_1, C_{2}$ in the linear systems $|\mco_{D_{3}}(5)|, |\mco_{D_{1}}(5)|,|\mco_{D_{2}}(5)|$  to get $\widetilde Y_1, \widetilde{Y}_2, \widetilde{Y}_3$ respectively. Let $\widetilde Y_{(ij)}$ be the proper transform of $Y_{(ij)}$ in $\widetilde {Y_{j}}$.  When we try to paste them, we have matching problems. For example, we need to paste $\widetilde Y_{(21)}$ in $\widetilde Y_{1}$ with $\widetilde Y_{(12)}$ in $\widetilde Y_{2}$.
Since the blow-up center $C_{3}$ lies on $Y_{(21)}$,  $\widetilde Y_{(21)}$ is isomorphic with $Y_{(21)}$ but
$C_{1}$ intersects with $Y_{(12)}$ at 15 points. So $\widetilde Y_{(12)}$ is the blow-up  of $Y_{(12)}$ of those 15 points and accordingly $\widetilde Y_{(12)}$ is not isomorphic to $\widetilde Y_{(21)}$.
Hence we cannot paste them.
This kind of problem does not occur in case of normal crossing varieties with only two components, where there is no triple locus.
We will show that one can still paste the varieties after blow-ups if one choose the blow-up center carefully and choose the order of blow-ups in some suitable manner.

Choose smooth curves $C_i$'s that intersect with $\tau = Y_{123}$ transversely, satisfying the condition:
\begin{align}
 C_{i} \cap \tau = C_{j} \cap \tau
\end{align}
for each $i, j$ (Figure \ref{fig1}). Note $C_i$ meets $\tau$ at 15 points but Figure \ref{fig1} is simplified.

\begin{figure}[h]
\begin{tikzpicture}[scale=1]
    \draw (-2,-.5) -- (0,0) -- (1.5,-1.3);
     \draw (0,0) -- (1.8, 1.3);
      \draw (-2,-4.5) -- (0,-4) -- (1.5,-5.3);
     \draw [dashed] (0,-4) -- (1.8, -2.7);
     \draw  (1.5,-2.92) -- (1.8, -2.7);
     \draw  (-2,-.5) --  (-2,-4.5);
     \draw [thick](0,0) -- (0,-4);
     \draw (1.8, -2.7) -- (1.8, 1.3);
     \draw (1.5,-1.3) --  (1.5,-5.3);
    \draw (-2, -2.5) .. controls (-1.5,-1.8) and (-1, -3) ..  (0,-2);
    \draw  (1.5, -3.3)  .. controls (1.3,-2.2) and (.3, -2.2) .. (0,-2);
    \draw [densely dashed] (0,-2) .. controls (.5,-2) and (.5, -1.5) .. (1.16,-1.01);
    \draw  (1.16,-1.01) .. controls (1.5,-.8) and (1.3, -.9) .. (1.8, -.7);
    \node [right ] at (1.1,-.5) {$C_1$};
     \node [right ] at (-1.7,-2.0) {$C_2$};
     \node [right ] at (.3,-2.5) {$C_3$};
     \node [right ] at (-.7,-4.8) {$Y_1$};
     \node [right ] at (1.7,-3.8) {$Y_2$};
     \node [right ] at (-.9,.7) {$Y_3$};
     \node [right ] at (-.4,-1.1) {$\tau$};
     \node [right ] at (-2.0,-.8) {$D_2$};
     \node [right ] at (1.15,.7) {$D_1$};
     \node [right ] at (.85,-4.7) {$D_3$};

%    \draw (4,0,0) -- (4,0,-3) -- (4,5,-3) -- (4,5,0) -- cycle;
%    \draw (4,5,0) -- (0,5,0) -- (0,5,-3) -- (4,5,-3);
%    \draw[style=dashed, color=gray] (4,0,-3) -- (0,0,-3)
%      -- (0,5,-3);
%    \draw[style=dashed, color=gray] (0,0,0) -- (0,0,-3);
%    \draw (2,-.4,0) node{4 ft};
%    \draw (4.6,-.2,-1.5) node{3 ft};
%    \draw (4.5,2.5,-3) node{5 ft};
  \end{tikzpicture}
  \caption{}
\label{fig1}
\end{figure}

For $\{i, j\}=\{1,2\}$, let $\pi_i: Y'_i \ra Y_i$ be the blow-up along $C_{j}$ on $D_j$, $Y'_{(3i)}$, $Y'_{(ji)}$ be the proper transform of  $Y_{(3i)}, Y_{(ji)}$ respectively  and $E_j$ be the exceptional divisor over $C_j$ (Figure \ref{fig2}).
\begin{figure}[h]
\begin{tikzpicture}[scale=1]
    \draw (-2,-.5) -- (0,0) -- (1.5,-1.3);
     \draw (0,0) -- (1.8, 1.3);
      \draw (-2,-4.5) -- (0,-4) -- (1.5,-5.3);
     \draw [dashed] (0,-4) -- (1.8, -2.7);
     \draw  (1.5,-2.92) -- (1.8, -2.7);
     \draw  (-2,-.5) --  (-2,-4.5);
     \draw [thick](0,0) -- (0,-4);
     \draw (1.8, -2.7) -- (1.8, 1.3);
     \draw (1.5,-1.3) --  (1.5,-5.3);
    \draw (-2, -2.5) .. controls (-1.5,-1.8) and (-1, -3) ..  (0,-2);
    \draw (-1.7, -2.76) .. controls (-1.2,-2.06) and (-.7, -3.26) ..  (0.3,-2.26); %%
    \draw (-2, -2.5)--(-1.7, -2.76);%%
    \draw (0,-2)--(0.3,-2.26); %%
  %  \draw  (1.5, -3.3)  .. controls (1.3,-2.2) and (.3, -2.2) .. (0,-2);
     \draw  (1.5, -3.3)  .. controls (1.3,-2.2) and (.7, -2.2) .. (0.2,-2.154);%%
    \draw [densely dashed] (0,-2) .. controls (.5,-2) and (.5, -1.5) .. (1.16,-1.01);
    \draw [densely dashed] (0.3,-2.26) .. controls (.8,-2.26) and (.8, -1.76) .. (1.46,-1.27);%%
    \draw  (1.16,-1.01) .. controls (1.5,-.8) and (1.3, -.9) .. (1.8, -.7);
    \draw  (1.46,-1.27) .. controls (1.8,-1.06) and (1.6, -1.16) .. (2.1, -.96); %%
    \draw (1.8, -.7)--(2.1, -.96);%%
    \node [right ] at (1.1,-.5) {$E_1$};
     \node [right ] at (-1.7,-2.0) {$E_2$};
     \node [right ] at (.6,-2.9) {$C'_3$};
     \node [right ] at (-.7,-4.8) {$Y'_1$};
     \node [right ] at (1.7,-3.8) {$Y'_2$};
     \node [right ] at (-.9,.7) {$Y_3$};
     \node [right ] at (-.5,-1.1) {$\tau'$};
%    \draw (4,0,0) -- (4,0,-3) -- (4,5,-3) -- (4,5,0) -- cycle;
%    \draw (4,5,0) -- (0,5,0) -- (0,5,-3) -- (4,5,-3);
%    \draw[style=dashed, color=gray] (4,0,-3) -- (0,0,-3)
%      -- (0,5,-3);
%    \draw[style=dashed, color=gray] (0,0,0) -- (0,0,-3);
%    \draw (2,-.4,0) node{4 ft};
%    \draw (4.6,-.2,-1.5) node{3 ft};
%    \draw (4.5,2.5,-3) node{5 ft};
  \end{tikzpicture}
  \caption{}
\label{fig2}
\end{figure}
Then $Y'_{(3i)}$ is isomorphic to $Y_{(3i)}$.
 Note $Y_{(3i)}$ is isomorphic with $Y_{(i3)}$.
So $Y'_{(3i)}$ ($\subset Y'_i$) is isomorphic to $Y_{(i3)}$ ($\subset Y_3$).
 Note that $Y'_{(ji)}$ ($\subset Y'_i$) is the blow-up  of $Y_{(ji)}$ at the points $C_i \cap \tau$.
 Since $C_1 \cap \tau= C_2 \cap \tau$,  $Y'_{(21)}$ ($\subset Y'_1$) and $Y'_{(12)}$ ($\subset Y'_2$) are isomorphic.
Let $C'_{3}$ be the proper transform of $C_{3}$ on $Y_{(21)}$ ($\subset Y_1$) in the blow-up $Y'_1 \ra Y_1$.
Since $C_{1} \cap \tau = C_{3} \cap \tau$  and  $Y'_{(21)}$ ($\subset Y'_1$) is the  blow-up  of $Y_{(21)}$ at the points $C_1 \cap \tau$, $C'_{3}$ does not meet with $\tau'$, where $\tau' = Y'_{(21)} \cap Y'_{(31)}$ and accordingly does not meet with $Y'_{(31)}$.
In sum, the curve $C'_{3}$ is disjoint with
$Y'_{(31)}$.
Let $\pi'_1: Y''_1 \ra  Y'_1$ be the blow-up along the curve $C'_{3}$ on $Y'_{(21)}$, $E'_2$ be the proper transform of $E_2$ and $E'_3$ be the exceptional divisor.
Then the blow-up $Y''_1 \ra  Y'_1$ does not change $ Y'_{(31)}$. i.e.\ the proper transform $Y''_{(31)}$ in $Y''_1$ of $Y'_{(31)}$ is isomorphic with  $Y'_{(31)}$ (Figure \ref{fig3}).
Since the curve $C'_{3}$ lies on $Y'_{(21)}$,  the proper transform $Y''_{(21)}$ in $Y''_1$ of $Y'_{(21)}$ is also isomorphic with  $Y'_{(21)}$.

\begin{figure}[h]
\begin{tikzpicture}[scale=1]
    \draw (-2,-.5) -- (0,0) -- (1.5,-1.3);
     \draw (0,0) -- (1.8, 1.3);
      \draw (-2,-4.5) -- (0,-4) -- (1.5,-5.3);
     \draw [dashed] (0,-4) -- (1.8, -2.7);
     \draw  (1.5,-2.92) -- (1.8, -2.7);
     \draw  (-2,-.5) --  (-2,-4.5);
     \draw [thick](0,0) -- (0,-4);
     \draw (1.8, -2.7) -- (1.8, 1.3);
     \draw (1.5,-1.3) --  (1.5,-5.3);
    \draw (-2, -2.5) .. controls (-1.5,-1.8) and (-1, -3) ..  (0,-2);
    \draw (-1.7, -2.76) .. controls (-1.2,-2.06) and (-.7, -3.26) ..  (0.3,-2.26); %%
    \draw (-2, -2.5)--(-1.7, -2.76);%%
    \draw (0,-2)--(0.3,-2.26); %%
  %  \draw  (1.5, -3.3)  .. controls (1.3,-2.2) and (.3, -2.2) .. (0,-2);
     \draw  (1.5, -3.3)  .. controls (1.3,-2.2) and (.7, -2.2) .. (0.2,-2.154);%%
     \draw (0.2,-2.154) -- (0.07,-2.25); %%%
     \draw (1.5, -3.25) -- (1.37, -3.396);%%%
     \draw  (1.37, -3.395)  .. controls (1.13,-2.296) and (.53, -2.184) .. (0.07,-2.25);%%%
    \draw [densely dashed] (0,-2) .. controls (.5,-2) and (.5, -1.5) .. (1.16,-1.01);
    \draw [densely dashed] (0.3,-2.26) .. controls (.8,-2.26) and (.8, -1.76) .. (1.46,-1.27);%%
    \draw  (1.16,-1.01) .. controls (1.5,-.8) and (1.3, -.9) .. (1.8, -.7);
    \draw  (1.46,-1.27) .. controls (1.8,-1.06) and (1.6, -1.16) .. (2.1, -.96); %%
    \draw (1.8, -.7)--(2.1, -.96);%%
    \node [right ] at (1.1,-.5) {$E_1$};
     \node [right ] at (-1.7,-2.0) {$E'_2$};
     \node [right ] at (.4,-2.9) {$E'_3$};
     \node [right ] at (-.7,-4.8) {$Y''_1$};
     \node [right ] at (1.7,-3.8) {$Y'_2$};
     \node [right ] at (-.9,.7) {$Y_3$};
     \node [right ] at (-.5,-1.1) {$\tau'$};
%    \draw (4,0,0) -- (4,0,-3) -- (4,5,-3) -- (4,5,0) -- cycle;
%    \draw (4,5,0) -- (0,5,0) -- (0,5,-3) -- (4,5,-3);
%    \draw[style=dashed, color=gray] (4,0,-3) -- (0,0,-3)
%      -- (0,5,-3);
%    \draw[style=dashed, color=gray] (0,0,0) -- (0,0,-3);
%    \draw (2,-.4,0) node{4 ft};
%    \draw (4.6,-.2,-1.5) node{3 ft};
%    \draw (4.5,2.5,-3) node{5 ft};
  \end{tikzpicture}
    \caption{}
\label{fig3}
\end{figure}
Now we can make a normal crossing variety  by pasting $Y''_1, Y'_2, Y_3$. We note that $Y''_{(21)} \simeq Y'_{(12)}$, $Y'_{(32)} \simeq Y_{(23)}$ and $Y_{(13)} \simeq Y''_{(31)}$. Moreover those isomorphism induce isomorphisms
$$Y''_{(21)} \cap Y''_{(31)} \simeq Y'_{(12)} \cap Y'_{(32)} \simeq Y_{(13)} \cap Y''_{(23)}.$$
Hence we can make a normal crossing variety $\widetilde Y$ such that
there is a normalization
$$\psi : Y''_1 \sqcup Y'_2 \sqcup Y_3 \ra \widetilde Y$$
 with $\psi(Y''_1) = \widetilde Y_1$,  $\psi(Y'_2) = \widetilde Y_2$,  $\psi(Y_3) = \widetilde Y_3$ --- this is already drawn in Figure \ref{fig3}. Let $\widetilde D_i = \widetilde Y_{jk}$ for $\{i,j,k\} = \{1,2,3\}$.
It is not hard to see that $\widetilde Y$ is $d$-semistable and projective (see also Theorem \ref{dsemi}, Theorem \ref{proj}).

By Theorem \ref{kana}, $\widetilde Y$ is smoothable to a Calabi--Yau threefold $M_{\widetilde Y}$.
%Now let us show that $\widetilde Y$ is projective.
%There is some positive integer $N_1$ such that the divisor $n_1 \pi_1^*(H_1) - E_2 $ is ample on %$Y'_1$
%if $n_1 \ge N_1$.
% There is some positive integer $N_2$ such that the divisor $n_2 (n_1\pi_1^*(H_1) - E'_2)-E'_3$ is %ample on $Y'_1 = \widetilde Y_1$ if $n_2 \ge N_2$.
%There is some positive integer $N_3$ such that the divisor $n_3 \pi_2^*(H_2) - E_1$ is ample on $Y'_2 %= \widetilde Y_2$ if $n_3 \ge N_3$.
By the formula in  Proposition \ref{eulereqn}, the topological Euler number of $M_{\widetilde Y}$ is
\begin{align*}
e(M_{\widetilde Y}) = \sum_i e(\widetilde Y_i) - 2 \sum_{i<j}e(\widetilde Y_{ij}) &+ 3 e( \widetilde Y_{123}).
\end{align*}
Note
\begin{align*}
e(\widetilde Y_1) &= e(Y''_1) = e(Y'_1) + e(C'_3) =  e(Y_1) + e(C_2) + e(C'_3)= e(Y_1) + e(C_2) + e(C_3),\\
e(\widetilde Y_2)& = e(Y'_2) = e(Y_2) + e(C_1),\,\, e(\widetilde Y_3) = e(Y_3),\,\, e(\widetilde Y_{12}) = e(Y'_{(12)}) = e(Y_{(12)}) +e(C_1 \cap \tau),\\
e(\widetilde Y_{23}) &= e(Y_{(23)}),\,\, e(\widetilde Y_{13}) = e(Y_{(13)}), \,\,e( \widetilde Y_{123}) =  e( Y_{123}).
\end{align*}
 Hence
\begin{align*}
e(M_{\widetilde Y}) &=       \sum_i e( Y_i) - 2 \sum_{i<j}e( Y_{ij}) + 3 e( Y_{123})
 + \sum_{i} e(C_{i}) - 2 e(C_1 \cap \tau) \\
    &= (4+4-6)-2(3+9+9)+3\cdot0 +(-10 -60 -60) -2 \cdot 15 \\
    &=-200.
\end{align*}
On the other hand,
one can show that  $\rank H^2(\widetilde Y, \mbz)=3$ and
$$( \pi'^*_1 (\pi^*_1(H_1)), {\pi^*_2}(H_2), H_3) $$
 belongs to $G^2(\widetilde Y, \mbz)$, where $H_i$ is the hyperplane section of
 $Y_i$.
Let $\hat H$ be the image in $H^2(\widetilde Y, \mbz)_f$ of  $( \pi'^*_1 (\pi^*_1(H_1)), {\pi^*_2}(H_2), H_3)$ by the map (\ref{gmap}), Since
$${\hat H}^3 =\pi'^*_1 (\pi^*_1(H_1))^3+ {\pi^*_2}(H_2)^3+ H_3^3 = 1+1+3=5$$
is non-zero, $\hat H$ is a non-zero element of $H^2(\widetilde Y, \mbz)_f$.
On the other hand, using Corollary \ref{hodge}, we have
$h^{2}(M_{\widetilde Y})  = 3-2=1$ (see also Lemma \ref{lem2}).
Hence $\{ \hat H \}$ is a basis for $H^2(M_{\widetilde Y}, \mbq)$.
Since the number ${\hat{H}}^3 = 5$ is positive and not a cube of an integer, $\hat H$ is the ample generator of $H^2(M_{\widetilde Y}, \mbz)_f$.
By Equation (\ref{chern}),

$$c_2(\widetilde Y) = \sum_i c_2 (\widetilde Y_i) - \sum_{i \neq j} \widetilde Y_{(ji)} \cdot c_1(\widetilde Y_{i}) = \sum_i c_2 (\widetilde Y_i) -\sum_{i}c_1 (\widetilde Y_{i})^2,$$
where we used $  c_1(\widetilde Y_{i}) = \sum_{j ( \neq i)} \widetilde Y_{(ji)}$.
Hence we have
\begin{align*}
\hat H \cdot c_2(M_{\widetilde Y}) &= \pi_1^*(H_1) \cdot (c_2(\widetilde Y_1)- c_1(\widetilde Y_1)^2) +\pi_2^*(H_2) \cdot (c_2(\widetilde Y_2) - c_1(\widetilde Y_2)^2) \\
&\,\,\,\,\,\,\,\,\,\,\,\,\,\,\,\,\,\,\,\,\,\,\,\,\,\,\,\,\,
\,\,\,\,\,\,\,\,\,\,\,\,\,\,\,\,\,\,\,\,\,\,\,\,\,\,\,\,\,\,\,\,\,\,
\,\,\,\,\,\,\,\,\,\,\,\,\,\,\,\,\,\,\,\,\,\,\,\,\,\,\,\,\,
+\pi_1^*(H_3) \cdot(c_2(\widetilde Y_3) - c_1(\widetilde Y_3)^2)\\
&= (26-(-4)) +(21-1)+ (12-12) = 50.
\end{align*}
So  $M_{\widetilde Y}$ is a Calabi--Yau threefold of Picard number one with invariants:
$$h^{1,1}(M_{\widetilde Y}) =1, h^{1,2}(M_{\widetilde Y}) =101, \hat H^3 = 5, \hat H \cdot c_2(M_{\widetilde Y}) = 50,$$
which  are invariants of a quintic Calabi--Yau threefold. This is expected since the smoothing is a quintic Calabi--Yau threefold in $\mbp^4$.

\section{Producing $d$-semistable models }  \label{sec5}
We apply the procedure in the previous section to more general   normal crossing variety $Y$, satisfying Condition \ref{refcond}.
We also want to consider the case that the blow-up curves $C_i$'s have multiple components since it gives us much more examples.

Before it, we need a notion of sequential blow-ups.
Let $Z$ be a smooth variety and $S$ be its smooth subvariety.  For a sequence of smooth divisors $c_1, c_2, \cdots, c_k$ on $S$.
We define the \emph{ sequential blow-up} $Z' \ra Z$ along $c_1, c_2, \cdots, c_k$ on $S$ as follows:
Let $Z^{(1)} \ra Z$ be the blow-up along $c_1$ and $S^{(1)}$ be the proper transform of $S$. Since the blow-up center $c_1$ lies on $S$, $S^{(1)}$ is isomorphic to $S$. So  $S^{(1)}$ contains copies of $c_1, c_2, \cdots, c_k$. We denote them by $c^{(1)}_1, c^{(1)}_2, \cdots, c^{(1)}_k$.
We construct $Z^{(2)}, Z^{(3)}, \cdots, Z^{(k)}$ inductively as follows.
Let $Z^{(l+1)} \ra Z^{(l)}$ be the blow-up along $c^{(l)}_{l+1}$ and $S^{(l+1)}$ be the proper transform of $S^{(l)}$. Since the blow-up center $c^{(l)}_{l+1}$ lies on $S^{(l)}$, $S^{(l+1)}$ is isomorphic to $S^{(l)}$. So  $S^{(l+1)}$ contains copies of $c^{(l)}_1, c^{(l)}_2, \cdots, c^{(l)}_k$. We denote them by $c^{(l+1)}_1, c^{(l+1)}_2, \cdots, c^{(l+1)}_k.$
Let $Z'=Z^{(k)}$, then the sequential blow-up is the composite $Z' \ra Z $ of the above blow-ups.
Let $S'$ be the proper transform of $S$, then $S'$ is isomorphic to $S$. Let $c'_i = c^{(k)}_i$, then $c'_i$ is isomorphic to $c_i$ for each $i$.
Let $E^{(l)}$ be the exceptional divisor  of the blow-up $Z^{(l)} \ra Z^{(l-1)}$ over $c^{l-1}_{l}$.
We denote the proper transform of  $E^{(l)}$ with respect to the map $Z' \ra  Z^{(l)}$ by $E_l$  and call it the exceptional divisor over the center $c_l$ of the sequential blow-up  $Z' \ra Z$.
Note  the sequential blow-up  depends on the order of blow-ups unless $S$ is a curve and so $c_i$'s are points.
We note some useful facts about  sequential blow-ups.
\begin{lemma} \label{lem1}
Let $S$ have codimension one in $Z$ and $\pi:Z' \ra Z$ be the  sequential blow-up along smooth divisors  $c_1, c_2, \cdots, c_k$ of $S$.
Then
\begin{enumerate}
\item \label{nd}  $$S'|_{S'} \sim (\pi|_S)^{*} (S|_S - c_1-c_2 - \cdots - c_k ).$$
\item \label{pj}Let $H$ be an ample divisor on $Z$, then there is some positive number $M$ such that, for any fixed positive integer $m$ with $m \ge M$,  the divisor
    $$H' := n \pi^* H -  E_k - m E_{k-1} - m^2 E_{k-2} - \cdots - m^{k-1} E_1 $$
    is ample on $Z'$ for sufficiently large $n$.
\item \label{td} Furthermore suppose that a normal crossing divisor $S + D_1 + \cdots + D_m$ is an anticanonical devisor of $Z$ such that any of $D_1,  \cdots,  D_m$ do not contain any of $c_1, c_2, \cdots, c_k$.
Then $S' + D'_1 + \cdots + D'_m$ is an anticanonical divisor of $Z'$, where $D'_i$ is the proper transform of $D_i$.
\end{enumerate}
\end{lemma}
\begin{proof}
Let $\pi^{(l)}:Z^{(l)} \ra Z^{(l-1)}$ be the blow-up along $c^{l-1}_{l}$. Then
\begin{align*}
S^{(l)}|_{S^{(l)}} &\sim {(\pi^{(l)}|_{S^{(l)}})}^{*} (S^{(l-1)}|_{S^{(l-1)}})  - c^{(l)}_{l}\\
&= {(\pi^{(l)}|_{S^{(l)}})}^{*} ( S^{(l-1)}|_{S^{(l-1)}} - c^{(l-1)}_{l}).
\end{align*}
Let $\pi^{[l]} = \pi^{(l)} \circ \pi^{(l-1)} \circ \cdots \circ \pi^{(1)}$, then one can inductively show
$$S^{(l)}|_{S^{(l)}} \sim {(\pi^{[l]}|_{S^{(l)}})}^{*} (S|_S - c_1-c_2 - \cdots - c_{l}).$$
By letting $l=k$, we have the first claim.

Consider fibers over points in $c^{l-1}_l$ in the blow-up     $\pi^{(l)}:Z^{(l)} \ra Z^{(l-1)}$ and let
 $N_l$ be the set of proper transforms in $Z'$ of those fibers. Then the relative cone $\NE(Z'/Z)$ of effective curves with respect to $\pi:Z' \ra Z$ is generated by curves in    $\bigcup_l N_l$.
 For any $L \in N_l$,
 $E_{l'} \cdot L =-1$ for $l'=l$,  $E_{l'} \cdot L =0$ for $l'<l$ and  $E_{l'} \cdot L \ge 0$ for $l'>l$.
For a divisor
$$D = -a_1 E_1 -a_2 E_2 - \cdots - a_k E_k,$$
 if $D \cdot L > 0 $ for any $L \in \bigcup_l N_l$, then $D$   is relatively ample with respect to the map $Z' \ra Z$.
Note that the set $\{ E_{l'} \cdot L | L \in N_l , l' < l \}$ has the maximum $\beta$.
Let $M  = \beta + 2$.
For a divisor
$$ D'=-  E_k - m E_{k-1} - m^2 E_{k-2} - \cdots - m^{k-1} E_1$$
with  $m>1$ and $L \in N_l$,
\begin{align*}
D' \cdot L &= -  E_k \cdot L - m E_{k-1} \cdot L - m^2 E_{k-2} \cdot L - \cdots - m^{k-l} E_l \cdot L \\
           &= -  E_k \cdot L - m E_{k-1} \cdot L - m^2 E_{k-2} \cdot L - \cdots + m^{k-l-1} E_{l-1} \cdot L   -m^{k-l} \\
           &\ge  - \beta - m \beta  - m^2 \beta  -\cdots - m^{k-l-1} \beta  + m^{k-l} \\
           &\ge m^{k-l} \left(1- \frac{\beta}{m-1} \right).
\end{align*}
So if $m \ge M$, then $D' \cdot L > 0$ and accordingly $D'$ is relatively ample with respect to the map $Z' \ra Z$. Hence we have the second claim.

Let $D_i^{(l)}$ be the proper transform in $Z^{(l)}$ of $D_i$ by the map $Z^{(l)} \ra Z$.
One can also inductively show
$$-K_{Z^{(l)}} \sim S^{(l)} + D^{(l)}_1 + \cdots + D^{(l)}_m$$
and by letting $l=k$ we have the third claim.

\end{proof}

From now on, we assume that the normal crossing variety $Y$ is three dimensional (Of course, we always assume that $Y$ satisfies Condition \ref{refcond}). Then its triple loci
$\tau$ is a curve.

Let $C_i=c_{1i}+c_{2i}+\cdots + c_{\alpha_i i}$  be  a divisor on $D_i$ for each $i$  that is a  sum of  distinct irreducible smooth curves $c_{ji}$'s on $D_i$. We require that  $C_i$'s  satisfy:
\begin{condition} \label{ccond}
Each $C_i$ intersects with $\tau$ transversely and
$$C_1 \cap \tau = C_2 \cap \tau =  C_3 \cap \tau.$$
\end{condition}

%\begin{remark}
%More generally one can assume that  $C_{ij}$ is a normal crossing divisor but one need to %be careful in choosing the order of blow-ups.
% In this note, we just assume that $C_{ij}$ is a smooth divisor for simplicity.
% Aso Condition \ref{connd2} can be replanced by scheme-theoretic intersection are the same.
%\end{remark}

If each $C_{i}$ is linearly equivalent to the divisor class $N_Y(D_i)$, defined in (\ref{cnormal}), we call $\{ C_1, C_2, C_3\}$ a collective normal divisor of $Y$.

For $\{i, j\}=\{1,2\}$, let $\pi_i: Y'_i \ra Y_i$ be the sequential blow-up along
$c_{1j}, c_{2j}, \cdots, c_{\alpha_j j}$  on $D_j$, $Y'_{(3i)}$ be the proper transform of  $Y_{(3i)}$ and $E_{lj}$ be the exceptional divisor over $c_{lj}$.
Then $Y'_{(3i)}$ is isomorphic to $Y_{(3i)}$.
 Note $Y_{(3i)}$ is isomorphic with $Y_{(i3)}$.
So $Y'_{(3i)}$ ($\subset Y'_i$) is isomorphic to $Y_{(i3)}$ ($\subset Y_3$).
 Note that $Y'_{(ji)}$ ($\subset Y'_i$) is the blow-up  of $Y_{(ji)}$ at the points $C_i \cap \tau$.
 Since $C_1 \cap \tau= C_2 \cap \tau$,  $Y'_{(21)}$ ($\subset Y'_1$) and $Y'_{(12)}$ ($\subset Y'_2$) are isomorphic.
Let $c'_{l3}$ be the proper transform of $c_{l3}$ on $Y_{(21)}$ ($\subset Y_1$) in the sequential blow-up $Y'_1 \ra Y_1$.
Since $C_{1} \cap \tau = C_{3} \cap \tau$  and  $Y'_{(21)}$ ($\subset Y'_1$) is the  blow-up  of $Y_{(21)}$ at the points $C_1 \cap \tau$, the effective divisor $C'_{3}:=\sum_l c'_{l3}$ does not meet with $\tau'$, where $\tau' = Y'_{(21)} \cap Y'_{(31)}$ and accordingly does not meet with $Y'_{(31)}$.
In sum, the divisor $C'_{3}$ is disjoint with
$Y'_{(31)}$.
Let $\pi'_1: Y''_1 \ra  Y'_1$ be the sequential blow-up along $c'_{13}, c'_{23}, \cdots, c'_{\alpha_3 3}$  on $Y'_{(21)}$, $E'_{l2}$ be the proper transform of $E_{l2}$ and $E'_{l3}$ be the exceptional divisor over $c'_{l3}$.
Then the sequential blow-up $Y''_1 \ra  Y'_1$ does not change $ Y'_{(31)}$. i.e.\ the proper transform $Y''_{(31)}$ in $Y''_1$ of $Y'_{(31)}$ is isomorphic with  $Y'_{(31)}$.
Since the curves $c'_{l3}$'s lie on $Y'_{(21)}$,  the proper transform $Y''_{(21)}$ in $Y''_1$ of $Y'_{(21)}$ is also isomorphic with  $Y'_{(21)}$.
Let us try to  make a normal crossing variety  by pasting $Y''_1, Y'_2, Y_3$. We note that $Y''_{(21)} \simeq Y'_{(12)}$, $Y'_{(32)} \simeq Y_{(23)}$ and $Y_{(13)} \simeq Y''_{(31)}$. Moreover those isomorphisms induce isomorphisms
$$Y''_{(21)} \cap Y''_{(31)} \simeq Y'_{(12)} \cap Y'_{(32)} \simeq Y_{(13)} \cap Y''_{(23)}.$$
Hence we can make a normal crossing variety $\widetilde Y$ such that
there is a normalization
$$\psi : Y''_1 \sqcup Y'_2 \sqcup Y_3 \ra \widetilde Y$$
 with $\psi(Y''_1) = \widetilde Y_1$,  $\psi(Y'_2) = \widetilde Y_2$,  $\psi(Y_3) = \widetilde Y_3$.
 We simply identify $Y''_1$ with $\widetilde Y_1$,  $Y'_2$ with $ \widetilde Y_2$ and  $Y_3$ with $\widetilde Y_3$.  Let $\widetilde D_i = \widetilde Y_{jk}$ for $\{i,j,k\} = \{1,2,3\}$.

 We also note that $Y'_{(21)} \simeq Y'_{(12)}$, $Y'_{(32)} \simeq Y_{(23)}$ and $Y_{(13)} \simeq Y'_{(31)}$. Hence we can make a normal crossing variety $\check Y$ by pasting $Y'_1, Y'_2, Y_3$
 ($\check Y_1 = Y'_1,\check Y_2 = Y'_2,\check Y_3 = Y_3$).
 Then there are natural maps
 $$\widetilde Y \stackrel{\pi'}{\ra} \check Y  \stackrel{\pi}{\ra}   Y.$$
 Note that $\pi|_{\check Y_i} =\pi_i$ for $i=1,2$ and $\pi'|_{\widetilde Y_1} = \pi'_1$.
 As before, let $\check D_i = \check Y_{jk}$ for $\{i,j,k\} = \{1,2,3\}$.

Since $\widetilde Y_i, \widetilde Y_{ij}$ are birational to $ Y_i,  Y_{ij}$ respectively, $\widetilde Y$ also satisfies (2) in Condition \ref{refcond}.

\begin{theorem} \label{dsemi}
 $\widetilde Y$ also satisfies (4) in Condition \ref{refcond} (i.e.\ it has trivial dualizing sheaf). Furthermore
If $\{C_1, C_2, C_3 \}$ is a collective normal class, then $\widetilde Y$ has trivial collective normal class.
\end{theorem}
\begin{proof}
 $Y_{(ji)} + Y_{(ki)}$ is an anticanonical normal crossing divisor of $Y_i$ for each $\{i,j,k \} = \{1,2,3 \}$ and so $\widetilde Y_{(13)} + \widetilde Y_{(23)}$ is  an anticanonical normal crossing divisor of $\widetilde Y_3$.
By (\ref{td}) in Lemma \ref{lem1}, $\check Y_{(ji)} + \check Y_{(ki)}$ is  an anticanonical normal crossing divisor of $\check Y_i$ for $i=1,2$ and
$\widetilde Y_{(21)} + \widetilde Y_{(31)}$ is  an anticanonical normal crossing divisor of $\widetilde Y_1$.
Since $\widetilde Y_2 = \check Y_2$, $\widetilde Y_{(12)} + \widetilde Y_{(31)}$ is  an anticanonical normal crossing divisor of $\widetilde Y_2$.
In sum,  $\widetilde Y_{(ji)} + \widetilde Y_{(ki)}$ is an anticanonical normal crossing divisor of  $\widetilde Y_i$ for any $\{i,j,k \} = \{1, 2, 3\}$ and accordingly  $\widetilde Y$ has trivial dualizing sheaf.

Now assume that $\{C_1, C_2, C_3 \}$ is a collective normal class of $Y$.
Firstly we show that the following claim.
\medskip

\emph{Claim.} $\{\check C_1, \check C_2, \check C_3 \}$ is a collective normal class of $\check Y$, where $\check C_1$, $\check C_2$ are the zero divisors on $\check Y_{23}$, $\check Y_{31}$ respectively and $\check C_3 = c'_{13}+c'_{23}+ \cdots + c'_{\alpha_3 3}$.

\medskip
\emph{Proof of Claim.}
Note $N_Y(Y_{i3}) =  Y_{(i3)}|_{Y_{i3}} + Y_{(3i)}|_{Y_{i3}}+  Y_{(ji3)} \sim C_j$ for $\{i, j\}=\{1,2\}$.
By (\ref{nd}) in Lemma \ref{lem1},
for $\{i, j\}=\{1,2\}$
$$\check Y_{(3i)}|_{\check Y_{(3i)}} \sim (\pi_i|_{\check Y_{i3}} )^* (Y_{(3i)}|_{Y_{(3i)}}  - c_{1j}- c_{2j}- \cdots - c_{\alpha_j j}) = (\pi_i|_{\check Y_{i3}} )^* (Y_{(3i)}|_{Y_{(3i)}}  - C_j).$$
And note $(\pi_i|_{\check Y_{i3}} )^*(Y_{(i3)}|_{Y_{i3}}) = \check Y_{(i3)}|_{\check Y_{i3}}$ and $(\pi_i|_{\check Y_{i3}} )^*( Y_{(ji3)}) =  \check Y_{(ji3)}$.
So we have
\begin{align*}
N_{\check Y}(\check Y_{i3}) &=  \check Y_{(i3)}|_{\check Y_{i3}} + \check Y_{(3i)}|_{\check Y_{i3}}+  \check Y_{(ji3)} \\
                  &\sim(\pi_i|_{\check Y_{i3}} )^*(Y_{(i3)}|_{Y_{i3}}) + (\pi_i|_{\check Y_{i3}} )^* (Y_{(3i)}|_{Y_{(3i)}})  - C_j)+  (\pi_i|_{\check Y_{i3}} )^*( Y_{(ji3)})\\
                  &=(\pi_i|_{\check Y_{i3}} )^*(Y_{(i3)}|_{Y_{i3}}) + Y_{(3i)}|_{Y_{(3i)}}  - C_j+  Y_{(ji3)})\\
                  &\sim(\pi_i|_{\check Y_{i3}} )^*(N_Y(Y_{i3}) -C_j)\\
                  &=(\pi_i|_{\check Y_{i3}} )^*(0)\\
                  &=0.
\end{align*}
Note
$$ (\pi_i|_{\check Y_{12}} )^* (Y_{(12)}|_{Y_{12}}) = \check Y_{(12)}|_{\check Y_{12}}, (\pi_i|_{\check Y_{12}} )^* (Y_{(21)}|_{Y_{12}}) = \check Y_{(21)}|_{\check Y_{12}}$$
and
$$(\pi_i|_{\check Y_{12}} )^* (Y_{(312)}) = \check Y_{(312)} + (\sum_l E_{li} )|_{\check Y_{12}}.$$
We have $(\sum_l E_{li}) |_{\check Y_{12}} = (\sum_l E_{lj})|_{\check Y_{12}}$ because of Condition \ref{ccond}.
Hence
\begin{align*}
N_{\check Y}(\check Y_{12}) &=  \check Y_{(12)}|_{\check Y_{12}} + \check Y_{(21)}|_{\check Y_{12}}+  \check Y_{(312)} \\
                  &\sim(\pi_i|_{\check Y_{12}} )^* (Y_{(12)}|_{Y_{12}}) + (\pi_i|_{\check Y_{12}} )^* (Y_{(21)}|_{Y_{12}}) +  (\pi_i|_{\check Y_{12}} )^* (Y_{(312)}) -(\sum_l E_{li} )|_{\check Y_{12}} \\
                  &=(\pi_i|_{\check Y_{12}} )^* (Y_{(12)}|_{Y_{12}}+Y_{(21)}|_{Y_{12}}) +  Y_{(312)}) -(\sum_l E_{li} )|_{\check Y_{12}}\\
                  &\sim(\pi_i|_{\check Y_{12}} )^* (C_3) - (\sum_l E_{li} )|_{\check Y_{12}}\\
                  &= \check C_3,
\end{align*}
which finishes the proof of claim.

Now let us complete the proof of the theorem.
$N_{\check Y}(\check Y_{23}) =0$ implies $N_{\widetilde Y}(\widetilde Y_{23}) =0$
because $\widetilde Y_2= \check Y_2$ and  $\widetilde Y_3= \check Y_3$.
$N_{\check Y}(\check Y_{13}) =0$ implies $N_{\widetilde Y}(\widetilde Y_{13}) =0$
Since the blow-up centers $c'_{l3}$'s are disjoint from $\check Y_{13}$.
Now it remains to show $N_{\widetilde Y}(\widetilde Y_{23}) =0$.
This can be showed similarly as before.
\begin{align*}
N_{\widetilde Y}(\widetilde Y_{12}) &=  \widetilde Y_{(12)}|_{\widetilde Y_{12}} + \widetilde Y_{(21)}|_{\widetilde Y_{12}}+  \widetilde Y_{(312)} \\
                  &\sim(\pi'_1|_{ \widetilde Y_{12}} )^*(\check Y_{(12)}|_{\check Y_{12}}) + (\pi'_1|_{\widetilde Y_{12}} )^* (\check Y_{(21)}|_{\check Y_{(12)}})  - C'_3)+  (\pi'_1|_{\widetilde Y_{12}} )^*( \check Y_{(312)})\\
                  &=(\pi'_1|_{ \widetilde Y_{12}} )^*(Y_{(12)}|_{Y_{12}}+ Y_{(21)}|_{Y_{(12)}}  - C'_3 +   \check Y_{(312)})\\
                  &=(\pi'_1|_{ \widetilde Y_{12}} )^*(N_Y(\check Y_{12}) -C'_3)\\
                  &=(\pi'_1|_{ \widetilde Y_{12}} )^*(0)\\
                  &=0.
\end{align*}

\end{proof}

To have the projectivity, we introduce another condition for $C_1, C_2, C_3$:

\begin{condition} \label{cond2}$\,$
\begin{enumerate}
\item $\alpha_1=\alpha_2=\alpha_3$ and
$$c_{l1}\cap \tau = c_{l2}\cap \tau = c_{l3}\cap \tau$$
for each $l$. Let $\alpha=\alpha_1$. \label{al}
\item For each  $l=1,2, \cdots, \alpha$ and distinct $i, j = 1, 2, 3$,
there is a divisor $G_{li}$ of $Y_i$ such that $G_{li} | _{D_j}$ is linearly equivalent to
$c_{lj}$. \label{am}
\end{enumerate}
\end{condition}

\begin{theorem} \label{proj}
If $C_i$'s satisfy Condition \ref{cond2}, then $\widetilde Y$ is projective.
\end{theorem}
\begin{proof}
Let $H_1, H_2, H_3$ be  ample divisors on $Y_1, Y_2, Y_3$ respectively such that $H_i|_{Y_{ij}} \sim H_j|_{Y_{ij}}$.
Then  by (\ref{pj}) in Lemma \ref{lem1},  for each $\{i, j\} = \{1,2 \}$, there is a positive number $m$ such that
$$\check H_i :=n \pi_i^*(H_i) - \sum_{l=1}^\alpha m^{\alpha -l} E_{lj}$$
is ample on $\check Y_i$  for   sufficiently large $n$.
Also the divisor
$$\check H_{3}:=n H_3 - \sum_l m^{\alpha -l} G_{l3} $$
is ample on $Y_3$ for sufficiently large $n$.
Hence one can choose some $n$ such that the divisor $\check H_i$ is ample on $\check Y_i$ for each $i=1,2,3$.
Firstly Condition \ref{ccond} implies $\check H_1 |_{\check D_3} \sim \check H_2 |_{\check D_3}$.
 The condition (\ref{am}) in Condition \ref{cond2} implies $\check H_i|_{\check D_j} \sim \check H_3|_{\check D_j}$ for distinct $i,j =  1, 2$. So we showed that $\check Y$ is projective.
Let $\check G_{l2} = \pi_2^*(G_{l2}) - E_{l1}$ for $l=1,2, \cdots, \alpha$.
Then
$$\check G_{l2}|_{\check D_3} \sim c'_{l3}.$$
Again,  by (\ref{pj}) in Lemma \ref{lem1}, there is some positive integer $m'$ such that
the divisors
$$\widetilde H_1 : = \tilde n {\pi'_1}^*(\check H_1) - \sum_{l=1}^\alpha m^{\alpha -l} E'_{l3}$$
and
$$\widetilde H_2 : = \tilde n \check H_2  -\sum_{l=1}^\alpha m^{\alpha -l} \check G_{l3}$$
that are ample on $\widetilde Y_1$ and $\widetilde Y_2$ respectively for sufficiently large $\tilde n$.
Let $\widetilde H_3 = \tilde n \check H_3$.
Then $\widetilde H_i$ is ample on $\widetilde Y_i$ and  $\widetilde H_i|_{\widetilde D_j} \sim \widetilde H_3|_{\widetilde D_j}$ for distinct $i,j =  1, 2$. So we showed that $\widetilde Y$ is also projective.

\end{proof}

\begin{lemma} \label{lem2}
If $C_i$'s satisfy Condition \ref{cond2},
$$h^{2}(\widetilde Y) = h^2(Y) + 2 \alpha.$$
\end{lemma}
\begin{proof}

The maps
$$\widetilde Y \stackrel{\pi'}{\ra} \check Y  \stackrel{\pi}{\ra}   Y.$$
induces pullbacks
$$H^2(Y_i, \mbq) \ra  H^2(\check Y_i, \mbq),  H^2(D_j, \mbq) \ra  H^2(\check D_j, \mbq)$$
and
$$ H^2(\check Y_i, \mbq) \ra  H^2(\widetilde Y_i, \mbq),  H^2(\check D_j, \mbq) \ra  H^2(\widetilde D_j, \mbq),$$
which altogether make the following commutative diagram.
$$
\xymatrix{
 H^2(\widetilde Y_1, \mbq) \oplus  H^2(\widetilde Y_2, \mbq) \oplus  H^2(\widetilde Y_3, \mbq) \ar[r]^{\tilde \mu}  & H^2(\widetilde D_1, \mbq) \oplus H^2(\widetilde D_2, \mbq) \oplus H^2(\widetilde D_3, \mbq)\\
 H^2(\check Y_1, \mbq) \oplus  H^2(\check Y_2, \mbq) \oplus  H^2(\check Y_3, \mbq) \ar[r]^{\check \mu}  \ar[u]^{g} & H^2(\check D_1, \mbq) \oplus H^2(\check D_2, \mbq) \oplus H^2(\check D_3, \mbq) \ar[u]\\
 H^2(Y_1, \mbq) \oplus  H^2(Y_2, \mbq) \oplus  H^2(Y_3, \mbq) \ar[r]^\mu \ar[u]^{f}  & H^2(D_1, \mbq) \oplus H^2(D_2, \mbq) \oplus H^2(D_3, \mbq) \ar[u]
}$$
Firstly by Proposition \ref{prop11}, $\ker \mu \simeq H^2(Y, \mbq)$, $\ker \check \mu \simeq H^2(\check Y, \mbq)$ and $\ker \tilde \mu \simeq H^2(\widetilde Y, \mbq)$. Note also that the maps $f, g$ are injective.
We can regard divisors as elements of $H^2(Y_i, \mbq), H^2(D_j, \mbq)$.
It is easy to check
$$E_l :=(E_{l2}, E_{l1}, G_{l3} ) \in \ker \check \mu.$$
Condition \ref{ccond} and the connectivity condition imply
 $$\ker \check \mu =  f ( \ker \mu ) \oplus \langle E_1, E_2, \cdots, E_\alpha \rangle   $$
 and so
 $$\dim (\ker \check \mu) = \dim ( f ( \ker \mu)) + \alpha =  h^2(Y) + \alpha.$$
It is also easy to check
 $$E'_l :=(E'_{l3}, \check G_{l2}, 0 ) \in \ker \tilde \mu.$$
Again the connectivity condition implies
 $$  \ker \tilde \mu = g ( \ker \check \mu) \oplus \langle E'_1, E'_2, \cdots, E'_\alpha \rangle $$
 and
 $$\dim \ker (\tilde \mu) = \dim (g(\ker \check \mu)) + \alpha = ( h^2(Y)  + \alpha) + \alpha = h^2(Y)+ 2 \alpha.$$
 So we are done.

\end{proof}

Combining Theorem \ref{dsemi}, \ref{proj},  we have the following corollary, which is the most important theorem in this note.
\begin{corollary} \label{refthmxxxxx}
Suppose  that $\{C_1, C_2, C_3 \}$ is a collective normal class of $Y$, satisfying Condition \ref{ccond}, \ref{cond2}.
Then $\widetilde Y$ also  satisfies Condition \ref{refcond} and it is $d$-semistable.  $\widetilde Y$ is smoothable to a Calabi--Yau threefold $M_{\widetilde Y}$ and hence it is a Calabi--Yau threefold of type III.
\end{corollary}

\begin{theorem} Assume all the conditions in Corollary \ref{refthmxxxxx} and let $M_{\widetilde Y}$ be the smoothing of $\widetilde Y$ with smooth total space, then
$$h^{1,1}(M_{\widetilde Y}) = h^2(Y) + 2 \alpha - 2$$
and
$$e(M_{\widetilde Y}) = \sum_i e(Y_i) - 2 \sum_j e(D_j)+ 3 e(\tau) +\sum_{i, l} e(c_{il}) -2 \gamma, $$
where $\gamma = e(\tau \cap C_1)$.
\end{theorem}
\begin{proof}
By Corollary \ref{hodge} and Lemma \ref{lem2}
$$h^{1,1}(M_{\widetilde Y}) = \dim(\ker \tilde \mu)-2  =h^2(Y) + 2 \alpha - 2.$$
By Proposition \ref{eulereqn},
\begin{align*}
e(M_{\widetilde Y}) &= e(\widetilde Y_1) + e(\widetilde Y_2) + e(\widetilde Y_3) - 2 \left(e(\widetilde D_1) + e(\widetilde D_2) + e(\widetilde D_3)\right) + 3e(\tilde \tau )\\
   &= (e(\check Y_1) + \sum_l e(c'_{l3})) +  e(\check Y_2)+e(Y_3)
   - 2 \left(e( D_1) + e( D_2) + (e( D_3) + \gamma)\right) + 3 e(\tau)\\
   &=((e(Y_1) + \sum_l e_(c_{l2})) + \sum_l e(c_{l3})) +  (e( Y_2) + \sum_l e(c_{l1}))+e(Y_3)\\
   &\,\,\,\,\,\,\,\,\,\,\,\,\,\,\,\,\,\,\,\,\,\,\,\,\,\,
   \,\,\,\,\,\,\,\,\,\,\,\,\,\,\,\,\,\,\,\,\,\,\,\,\,\,\,
   \,\,\,\,\,\,\,\,\,\,\,\,\,\,\,\,\,\,\,\,\,\,\,\,\,\,\,\,  - 2 \left(e( D_1) + e( D_2) + (e( D_3) + \gamma)\right) + 3 e(\tau)\\
   &=e( Y_1) + e( Y_2) + e( Y_3) - 2 \left(e( D_1) + e( D_2) + e( D_3)\right) + 3e( \tau ) + \sum_{i, l} e(c_{il}) -2 \gamma.
\end{align*}

\end{proof}

Since $h^{1,2} (M_{\widetilde Y})= h^{1,1}(M_{\widetilde Y})- \frac{1}{2} e(M_{\widetilde Y})$, we have determined all the Hodge numbers of $M_{\widetilde Y}$.

Now return to Example \ref{quintic}. One can choose curves
$$C_i = c_{1i}+ c_{2i}+ \cdots + c_{\alpha i} $$
 with $c_{li}$ belonging to the linear system $| \mco_{D_i}(a_l) |$ such that
$C_1, C_2, C_3$ satisfies Condition \ref{ccond}, Condition \ref{cond2} and the conditions in Lemma \ref{lem2}.
If $a_1 + a_2 + \cdots + a_\alpha = 5$, then $\{C_1, C_2, C_3 \}$ is a collective normal curve of $Y$.
As described, we can build a normal crossing variety $\widetilde Y$ with respect to $C_i$'s, where we let $G_{li} = \mco_{Y_i}(a_l)$. Then $\widetilde Y$ is projective and $d$-semistable and so it is smoothable to a Calabi--Yau threefold $M_{\widetilde Y}$. Noting $h^2(Y)=1$ and $\gamma = 15$, we have

$$h^{1,1}(M_{\widetilde Y}) = 2 \alpha-1,$$
\begin{align*}
e(M_{\widetilde Y}) & = -70+ \sum_l(3-a_l)a_l + 2 \sum 3(1-a_l)a_l\\
                    &= -25 -7 \sum_l a_l^2
\end{align*}
and so
$$h^{1,2}(M_{\widetilde Y}) = h^{1,1}(M_{\widetilde Y}) - \frac{1}{2}e(M_{\widetilde Y}) = \frac{1}{2}(4 \alpha +  7 \sum_l a_l^2 + 23).$$

Note that Hodge numbers of $M_{\widetilde Y}$ depend on $a_l$'s but not on the ordering of $a_l$'s.
Let us take some examples. For the both cases of $( a_1, a_2 )=(1,4)$ and $(4,1)$, we have the same Hodge numbers are $(h^{1,1}, h^{1,2}) = (3, 75)$.
But the sequential blow-up depends on the ordering of blow-ups and so one can expect that smoothed Calabi--Yau threefolds are different.
Let $\widetilde Y^{14}$ and $\widetilde Y^{41}$ be normal crossing varieties for $( a_1, a_2)=(1,4)$ and $(4,1)$ respectively.
By using the technics in \S 6, \S 7 in \cite{Lee}, one can show that the ternary cubic forms on $H^2(M_{\widetilde Y^{14}}, \mbz)$, $H^2(M_{\widetilde Y^{41}}, \mbz)$ are different.
%So $M_{\widetilde Y^{41}}$ and $M_{\widetilde Y^{41}}$ are non-homeomorphic Calabi--Yau threefolds indeed.
There is another variation. Let $Y_1$ be a cubic threefold in $\mbp^4$ and $Y_2, Y_3$ be hyperplanes in $\mbp^4$ respectively. Then $Y=Y_1 \cup Y_2 \cup Y_3$ is the same normal crossing with the previous one and only the index ordering is changed. However if we build $\widetilde Y$, it is different from the previous one because $Y_1$ is sequentially blow up twice and $Y_2$ is sequentially blow up once --- choosing different index orders of $Y_i$'s changes $\widetilde Y$ and gives us non-homeomorphic Calabi--Yau threefolds $M_{\widetilde Y}$'s with same Hodge numbers.
In this way, one can build several non-homeomorphic Calabi--Yau threefolds with same Hodge numbers.
For the rest of examples, we will give only Hodge numbers of $M_{\widetilde Y}$'s:

\begin{center}
    \begin{tabular}{| c || c | c | c |c | c | c | c |}
    \hline
    $( a_1, a_2,  \cdots, a_\alpha )$ &(1,1,1,1,1)& (1,1,1,2 )&  (1,1,3)&  (1,4 )&  (1,2,2 )&  (2,3 )&  ( 5) \\
    \hline
    $(h^{1,1}, {h^{1,2}}^{\,})$ & (9, 39 )&  (7, 44)&  (5, 56)&  (3, 75)&  (5, 49)&  $(3, 61)^*$&  (1, 101)\\
    \hline
    \end{tabular}
\end{center}

In the table, `$\,^*$' means (also will mean later) that such Hodge pairs do not overlap with those of Calabi--Yau threefolds in toric construction (\cite{Al, KrSk1, KrSk2}) or examples recently constructed in \cite{Lee1, Lee11}.
The Hodge pair $(1, 101)$ comes from the quintic threefold in $\mbp^4$ which is a toric variety.
Except for this case, there are multiple non-homeomorphic Calabi--Yau threefolds having the Hodge numbers as previously explained and  although other pairs of Hodge numbers appear in the toric construction, probably those Calabi--Yau threefolds are different from ones of same Hodge numbers in the toric construction --- there seems no reason that they are the same ones. One possible way of distinguishing them is comparing the cubic forms on their second integral cohomology classes.

\section{More examples}  \label{sec6}
Applying technics in the previous sections, we construct more $d$-semistable Calabi--Yau threefolds of type III. In each of the following examples, we start with a three-dimensional normal crossing variety $Y= Y_1 \cup Y_2 \cup Y_3$  with  a divisor
$C_i = c_{1i} +c_{2i} +···+ c_{\alpha i} $on $D_i$  for $i=1, 2 , 3$ such that
\begin{enumerate}
\item $Y$ satisfies Condition \ref{refcond},
\item  $c_{ji}$ ’s are distinct irreducible smooth curves on $D_i$ and
\item $\{ C_1 ,C_2 , C_3 \}$ is a collective normal class of $Y$, satisfying Condition \ref{ccond}, \ref{cond2}.
\end{enumerate}

As we did in the previous section, we make a new normal crossing $\widetilde Y$, by pasting sequential blow-ups of $Y_i$'s with respect to the curves $c_{ij}$'s, which is now smoothable to a Calabi--Yau threefold $M_{\widetilde Y}$. We give the Hodge numbers of $M_{\widetilde Y}$'s in each example.
\begin{example}

Let $S, S'$ be smooth quadratic hypersurfaces in $\mbp^3$ that intersect each other transversely.
We note that $S, S'$ are  projectively isomorphic. So there is an isomorphism $\phi:\mbp^3 \ra \mbp^3$ such that $\phi(S) = S'$. Let $\phi':S \ra S'$ be the restriction of $\phi$ to $S$.
Let $Y_1, Y_2, Y_3$ be copies of $\mbp^3$ and $S_i, S'_i$ be copies in $Y_i$ of $S, S'$.
Identifying $S_1$ with $S'_2$, $S_2$ with $S'_3$ and $S_3$ with $S'_1$ via copies of the isomorphism $\phi'$, we have a normal crossing variety $Y = Y_1 \cup Y_2 \cup Y_3$.
It is easy to check that $Y$ is projective and that have a trivial dualizing sheaf.
$(\mco_{D_1}(6), \mco_{D_2}(6), \mco_{D_2}(6))$ is the collective normal class of $Y$.
  One can choose curves $C_i = c_{1i}+ c_{2i}+ \cdots + c_{\alpha i} $ with $c_{li}$ belonging to the linear system $\mco_{D_i}(a_l)$ such that
$C_1, C_2, C_3$ satisfies Condition \ref{ccond}, Condition \ref{cond2} and the conditions in Lemma \ref{lem2}.
If $a_1 + a_2 + \cdots + a_\alpha = 6$, then $\{C_1, C_2, C_3 \}$ is a collective normal curve of $Y$.
As described, we build a normal crossing variety $\widetilde Y$ with respect to $C_i$'s. Then $\widetilde Y$ is projective and $d$-semistable and so it is smoothable to a Calabi--Yau threefold $M_{\widetilde Y}$. Its dual complex is two-dimensional, so it is a $d$-semistable Calabi--Yau threefold of type III.
Noting $h^2(Y)=1$ and $\gamma = 24$, we have

$$h^{1,1}(M_{\widetilde Y}) = 2 \alpha-1,$$
\begin{align*}
e(M_{\widetilde Y}) & = -60+ \sum_l 3 \cdot 2 (2-a_l)a_l = 12 -6 \sum_l a_l^2
\end{align*}
and so
$$h^{1,2}(M_{\widetilde Y}) = h^{1,1}(M_{\widetilde Y}) - \frac{1}{2}e(M_{\widetilde Y}) = \frac{1}{2}(4 \alpha +  6 \sum_l a_l^2 - 14).$$
The Hodge number numbers
of $M_{\widetilde Y}$ are as follows:
\begin{center}
  \begin{tabular}{ |l | l| }
    \hline
    $( a_1, a_2,  \cdots, a_\alpha )$ & $(h^{1,1}, h^{1,2})$ \\ \hline \hline
    (1,1,1,1,1,1)& (11,23)\\ \hline
    (1,1,1,1,2 )& (9,27)\\ \hline
    (1,1,1,3)&  (7,37)\\ \hline
    (1,1,2,2)& (7,31)\\ \hline
    (1,1,4 )& (5,53)\\ \hline
     (1,2,3 )& (5,41)\\ \hline
      (2,2,2 )& (5,35)\\ \hline
     (1,5)& (3,75)\\ \hline
      (2,4)&(3,57)\\ \hline
      (3,3)&(3,51)\\ \hline
     (6) &(1,103)\\ \hline
  \end{tabular}
\end{center}

\end{example}

\begin{example}

Now consider  a normal crossing $Y$ of  smooth hypersurfaces $Y_1, Y_2$ and $Y_3$ in a smooth quartic hypersurface in $\mbp^5$ with degrees $1,1$, and $2$ respectively.
Clearly $Y$ is projective and that have a trivial dualizing sheaf.
$(\mco_{D_1}(4), \mco_{D_2}(4), \mco_{D_2}(4))$ is the collective normal class of $Y$.
  One can choose curves $C_i = c_{1i}+ c_{2i}+ \cdots + c_{\alpha i} $ with $c_{li}$ belonging to the linear system $\mco_{D_i}(a_l)$ such that
$C_1, C_2, C_3$ satisfies Condition \ref{ccond}, Condition \ref{cond2} and the conditions in Lemma \ref{lem2}.
If $a_1 + a_2 + \cdots + a_\alpha = 3$, then $\{C_1, C_2, C_3 \}$ is a collective normal curve of $Y$.
Build a normal crossing variety $\widetilde Y$ with respect to $C_i$'s. Then $\widetilde Y$ is a $d$-semistable Calabi--Yau threefold of type III. Note $h^2(Y)=1$ and $\gamma = 16$.
The Hodge number numbers
of $M_{\widetilde Y}$ are as follows:
\begin{center}
  \begin{tabular}{ |l | l| }
    \hline
    $( a_1, a_2,  \cdots, a_\alpha )$ & $(h^{1,1}, h^{1,2})$ \\ \hline \hline
  (1, 1, 1, 1)& (7,35)\\ \hline
    (1, 1, 2)& (5,43)\\ \hline
    (1, 3)&  $(3,61)^*$\\ \hline
    (2, 2)& (3,51)\\ \hline
    (4)& $(1,89)^*$\\ \hline
  \end{tabular}
\end{center}

\end{example}

\begin{example}

Let $Y$ be  a normal crossing of  smooth hypersurfaces $Y_1, Y_2$ and $Y_3$ in a smooth cubic hypersurface in $\mbp^5$ with degree three.
$(\mco_{D_1}(3), \mco_{D_2}(3), \mco_{D_2}(3))$ is the collective normal class of $Y$.
  Choose curves $C_i = c_{1i}+ c_{2i}+ \cdots + c_{\alpha i} $ with $c_{li}$ belonging to the linear system $\mco_{D_i}(a_l)$ such that $a_1 + a_2 + \cdots + a_\alpha = 3$, then $\{C_1, C_2, C_3 \}$ is a collective normal curve of $Y$.
Build a normal crossing variety $\widetilde Y$ with respect to $C_i$'s. Then $\widetilde Y$ is a $d$-semistable Calabi--Yau threefold of type III. Note $h^2(Y)=1$ and $\gamma = 9$.
The Hodge number numbers
of $M_{\widetilde Y}$ are as follows:
\begin{center}
  \begin{tabular}{ |l | l| }
    \hline
    $( a_1, a_2,  \cdots, a_\alpha )$ & $(h^{1,1}, h^{1,2})$ \\ \hline \hline
  (1, 1, 1)& (5,50)\\ \hline
    (1, 2)& (3,57)\\ \hline
    (3)&  $(1,73)^*$\\ \hline
  \end{tabular}
\end{center}

\end{example}

\begin{example}

Let $Y$ be  a normal crossing of  smooth hypersurfaces $Y_1, Y_2$ and $Y_3$ of degree one in a smooth complete intersections two quadrics in $\mbp^6$.
$(\mco_{D_1}(3), \mco_{D_2}(3), \mco_{D_2}(3))$ is the collective normal class of $Y$.
  Choose curves $C_i = c_{1i}+ c_{2i}+ \cdots + c_{\alpha i} $ with $c_{li}$ belonging to the linear system $\mco_{D_i}(a_l)$ such that  $a_1 + a_2 + \cdots + a_\alpha = 3$ and build a normal crossing variety $\widetilde Y$ with respect to $C_i$'s. Then $\widetilde Y$ is a $d$-semistable Calabi--Yau threefold of type III.
  Note $h^2(Y)=1$ and $\gamma = 12$.
  The Hodge number numbers
of $M_{\widetilde Y}$ are as follows:
\begin{center}
  \begin{tabular}{ |l | l| }
    \hline
    $( a_1, a_2,  \cdots, a_\alpha )$ & $(h^{1,1}, h^{1,2})$ \\ \hline \hline
  (1, 1, 1)& (5,41)\\ \hline
    (1, 2)& (3,51)\\ \hline
    (3)&  (1,73)\\ \hline
  \end{tabular}
\end{center}

\end{example}

\begin{example}

Let $Y$ be  a normal crossing of  smooth hypersurfaces $Y_1, Y_2$ and $Y_3$ of degree one in a section of the Grassmannian $\Gr(2,5)$ embedded by Pl\"ucker by a subspace of codimension two.
$(\mco_{D_1}(3), \mco_{D_2}(3), \mco_{D_2}(3))$ is the collective normal class of $Y$.
  Choose curves $C_i = c_{1i}+ c_{2i}+ \cdots + c_{\alpha i} $ with $c_{li}$ belonging to the linear system $\mco_{D_i}(a_l)$ such that $a_1 + a_2 + \cdots + a_\alpha = 3$ and build a normal crossing variety $\widetilde Y$ with respect to $C_i$'s. Then $\widetilde Y$ is projective and $d$-semistable and so it is smoothable to a Calabi--Yau threefold $M_{\widetilde Y}$. Note $h^2(Y)=1$ and $\gamma = 15$.
  The Hodge number numbers
of $M_{\widetilde Y}$ are as follows:
\begin{center}
  \begin{tabular}{ |l | l| }
    \hline
    $( a_1, a_2,  \cdots, a_\alpha )$ & $(h^{1,1}, h^{1,2})$ \\ \hline \hline
  (1, 1, 1)& (5,35)\\ \hline
    (1, 2)& $(3,48)^*$\\ \hline
    (3)&  $(1,76)^*$\\ \hline
  \end{tabular}
\end{center}
\end{example}

\begin{example}

Let $Y$ be  a normal crossing of  smooth hypersurfaces $Y_1, Y_2$ and $Y_3$ in $\mbp^2 \times \mbp^2$ with bi-degree $(1,1)$.
$(\mco_{D_1}(3,3), \mco_{D_2}(3,3), \mco_{D_2}(3,3))$ is the collective normal class of $Y$.
  One can choose curves $C_i = c_{1i}+ c_{2i}+ \cdots + c_{\alpha i} $ with $c_{li}$ belonging to the linear system $\mco_{D_i}(a_l, b_l)$ such that
$C_1, C_2, C_3$ satisfies Condition \ref{ccond}, Condition \ref{cond2} and the conditions in Lemma \ref{lem2}.
If $a_1 + a_2 + \cdots + a_\alpha = 3$ and $b_1 + b_2 + \cdots + b_\alpha = 3$, then $\{C_1, C_2, C_3 \}$ is a collective normal curve of $Y$.
As described, we build a normal crossing variety $\widetilde Y$ with respect to $C_i$'s. Then $\widetilde Y$ is projective and $d$-semistable and so it is smoothable to a Calabi--Yau threefold $M_{\widetilde Y}$. Note $h^2(Y)=2$ and $\gamma = 18$.
Note
$$e(Y_i) = 6, e(D_i) = 6.$$
So $\sum_i e(Y_i) -2 \sum_i e(D_i) -2 \gamma = 18 - 4 \cdot 18 = -54.$
We have $h^{1,1}(M_{\widetilde Y}) = 2 \alpha$,
\begin{align*}
e(M_{\widetilde Y}) & = -54+ \sum_l 3 ( 3 a_l - a_l^2 + 3 b_l - b_l^2 - 4 a_l b_l ) \\
                    &= -  3\sum_l ( a_l^2 + b_l^2 + 4 a_l b_l )
\end{align*}
and so
$$h^{1,2}(M_{\widetilde Y}) = h^{1,1}(M_{\widetilde Y}) - \frac{1}{2}e(M_{\widetilde Y}) = \frac{1}{2}(4 \alpha +  6 \sum_l a_l^2 - 14).$$

The Hodge number numbers
of $M_{\widetilde Y}$ are as follows:
\begin{center}
$  \begin{array}{ |l | l| }
    \hline
   ( (a_1, b_1), (a_2, b_2),  \cdots , (a_\alpha, b_\alpha ) )& (h^{1,1}, h^{1,2}) \\ \hline \hline
   \left((3, 3)\right)&(2,83)\\ \hline
\left((1, 0), (2, 3)\right)&(4,61)\\ \hline
\left((1, 3), (2, 0)\right)&(4,43)^*\\ \hline
\left((0, 3),(3, 0)\right)&(4,31)^*\\ \hline
\left((0,   2), (3, 1)\right)&(4,43)^*\\ \hline
\left((1, 2), (2, 1)\right)&(4,43)^*\\ \hline
\left((1, 1), (2, 2)\right)&(4,49)\\ \hline
\left((0, 1), (3, 2)\right)&(4,61)\\ \hline
\left((1, 0), (1, 0), (1, 3)\right)&(6,42)\\ \hline
\left((0, 3), (1, 0), (2, 0)\right)&(6,27)\\ \hline
\left((0,  2), (1, 0), (2, 1)\right)&(6,33)\\ \hline
\left((1, 0), (1, 1), (1, 2)\right)&(6,36)\\ \hline
\left((0, 1), (1, 0), (2,2)\right)&(6,45)\\ \hline
\left((0, 2), (1, 1), (2,0)\right)&(6,27)\\ \hline
\left((0, 1), (1, 2), (2, 0)\right)&(6,33)\\ \hline
\left((0,1), (0, 2), (3, 0)\right)&(6,27)\\ \hline
\left((0, 1), (0, 1), (3, 1)\right)&(6,42)\\ \hline
\left((0, 1), (1, 1), (2, 1)\right)&(6,36)\\ \hline
\left((1, 1), (1, 1), (1, 1)\right)&(6,33)\\ \hline
\left((0, 3), (1, 0), (1, 0), (1, 0)\right)&(8,26)\\ \hline
\left((0, 2), (1, 0), (1, 0), (1, 1)\right)&(8,26)\\ \hline
\left((0, 1), (1, 0), (1, 0), (1, 2)\right)&(8,32)\\ \hline
\left((0, 1), (0, 2), (1, 0), (2, 0)\right)&(8,23)^*\\ \hline
\left((0, 1), (0, 1), (1,0), (2, 1)\right)&(8,32)\\ \hline
\left((0, 1), (1, 0), (1, 1), (1, 1)\right)&(8,29)\\ \hline
\left((0, 1), (0, 1), (1, 1), (2, 0)\right)&(8,26)\\ \hline
\left((0, 1), (0, 1), (0, 1), (3, 0)\right)& (8,26)\\ \hline
\left((0, 1), (0, 2), (1, 0), (1, 0), (1, 0)\right)&(10,22)\\ \hline
\left((0, 1), (0, 1), (1, 0), (1, 0), (1, 1)\right)&(10,25)\\ \hline
\left((0, 1), (0, 1), (0, 1), (1, 0), (2, 0)\right)&(10,22)\\ \hline
\left((0, 1), (0, 1), (0,  1), (1, 0), (1, 0), (1, 0)\right)&(12,21)\\ \hline
  \end{array}
 $
\end{center}

\end{example}

In general, let $Z$ be a  projective Gorenstein variety  of dimension four. Let $Y=Y_1\cup Y_2 \cup Y_3$ be normal crossing hypersurface in $Z$ such that $Y_1 + Y_2 + Y_3 \sim -K_Z$.
Then $Y$ is projective and have trivial dualizing sheaf. We assume that $Y$ satisfies the homological condition (3) in Theorem \ref{kana}.
Suppose that one can choose collective normal curves $C_i = c_{1i}+ c_{2i}+ \cdots + c_{\alpha i}$ satisfying Condition \ref{ccond}, Condition \ref{cond2} and the conditions in Lemma \ref{lem2}.
Then we can build a normal crossing variety $\widetilde Y$ with respect to $C_i$'s, which is a $d$-semistable Calabi--Yau threefold of type III.
 
Now let us consider some examples of $d$-semistable Calabi--Yau threefolds of type III that consists of more than three components.
Let $Y=Y_1 \cup Y_2 \cup Y_3$ be any $d$-semistable Calabi--Yau threefold of type III  --- we already constructed several of such examples.
Then there is a semistable degeneration $\mcX \ra \Delta$  of Calabi--Yau threefolds whose central fiber $\mcX_0$ is $Y$. Let $\tau$ be the triple curve in $Y$ and $\mcX' \ra \mcX$ be the blow-up along $\tau$. Now our new degeneration $\mcX' \ra \Delta$ is not semistable because the central fiber is not reduced. However a base extension $\Delta' \ra \Delta $ by $t \mapsto t^3$ gives us a semistable degeneration $\mcX' \times_\Delta \Delta' \ra \Delta'$ ( Mumford's semistable reduction). Now the central fiber is a normal crossing variety
$$Y'=Y'_1 \cup Y'_2 \cup Y'_3 \cup F,$$
where $F$ is the exceptional devisor of the blow-up.
Note that $F$ is a $\mbp^1$-bundle over $\tau.$
The dual complex of $Y'$ has three triangles as its maximal cells and so $Y'$ is $d$-semistable Calabi--Yau threefolds of type III.
By doing base changes after blowing up the triple loci of the central fiber, one can add as many  components to central fiber as one wants.
All the newly added components are $\mbp^1$-bundles over elliptic curves.
Note a $\mbp^1$-bundle over elliptic curve does not satisfy the cohomological condition (3) in Theorem \ref{kana} ---($h^1(F, \mco_F) \neq 0$).
So those $d$-semistable Calabi--Yau threefolds do not fit into the situation in Theorem \ref{kana}. This may imply that Theorem \ref{kana} needs some generalization.

The author is very thankful to the  referee for making several valuable suggestions for the initial draft of this note.
This work was supported by  Basic Science Research Program
through the National Research Foundation of Korea(NRF) funded by the Ministry of Education (NRF-2017R1D1A2B03029525) and  Hongik University Research Fund.

%%%%%%%%%%%%%% reference %%%%%%%%%%%%%%%%%

\end{document}